\documentclass[a4, 11pt]{article}%
\usepackage{epsfig,subfig,amsmath, amsthm, amscd, amsfonts, amssymb,
graphicx,  color,colortbl,lscape,natbib}
\usepackage[bookmarksnumbered,pagebackref,plainpages,colorlinks,
linkcolor=red, citecolor=blue, citebordercolor={0 0 1}]{hyperref}
\usepackage{amsmath}
\usepackage{amsfonts}
\usepackage{amssymb}
\usepackage{graphicx}%
\providecommand{\U}[1]{\protect\rule{.1in}{.1in}}
\providecommand{\U}[1]{\protect\rule{.1in}{.1in}}
\headheight\baselineskip\headsep2\baselineskip\textwidth210mm
\advance\textwidth -2in\textheight287mm\advance\textheight-2in
\newtheorem{theorem}{Theorem}

\newtheorem{corollary}[theorem]{Corollary}
\theoremstyle{definition}
\newtheorem{definition}[theorem]{Definition}
\newtheorem{example}[theorem]{Example}

\newtheorem{remark}[theorem]{Remark}

\let\oldendproof\endproof

\def\endproof{\hfill$\blacksquare$\oldendproof}
\begin{document}

\title{\textbf{Parameter estimation based on cumulative Kullback-Leibler divergence}}
\author{Yaser Mehrali$^{1}$, and Majid Asadi$^{2}$\\{\footnotesize Department of Statistics, University of Isfahan, Isfahan,
81744, Iran}\\{\footnotesize E-mail: yasermehrali@gmail.com}$^{1},$%
{\footnotesize \ m.asadi@sci.ui.ac.ir}$^{2}$}
\maketitle

\begin{abstract}
In this paper, we propose some estimators for the parameters of a statistical
model based on Kullback-Leibler divergence of the survival function in
continuous setting. We prove that the proposed estimators are subclass of
"generalized estimating equations\rq\rq estimators. The asymptotic properties
of the estimators such as consistency, asymptotic normality, asymptotic
confidence interval and asymptotic hypothesis testing are investigated.
\newline



\end{abstract}


\section{Introduction}

The Kullback-Leibler ($KL$) divergence or relative entropy is a measure of
discrimination between two probability distributions. If $X$ and $Y$ have
probability density functions $f$ and $g$, respectively, the $KL$ divergence
of $f$ relative to $g$ is defined as%
\[
D\left(  f||g\right)  =\int\limits_{%
\mathbb{R}
}f\left(  x\right)  \log\frac{f\left(  x\right)  }{g\left(  x\right)  }dx,
\]

for $x$ such that $g(x)\not =0$. The function $D\left(  f||g\right)  $ is
always nonnegative and it is zero if and only if $f=g$ $a.s.$.

Let $f\left(  x;\mbox{\boldmath
$\theta$}\right)  $ belong to a parametric family with $k$-dimensional
parameter vector $\mbox{\boldmath
$\theta$}\in\mathbf{\Theta\subset%
\mathbb{R}
}^{k}$ and $f_{n}$ be kernel density estimation of $f$ based on $n$ random
variables $\{X_{1},\ldots,X_{n}\}$ of distribution $X$.
\cite{Basu:Lindsay:1994} used $KL$ divergence of $f_{n}$ relative to $f\,$\ as%
\begin{equation}
D\left(  f_{n}||f\right)  =\int f_{n}\left(  x\right)  \log\frac{f_{n}\left(
x\right)  }{f\left(  x;\mbox{\boldmath
$\theta$}\right)  }dx, \label{eq1.1}%
\end{equation}
and defined the minimum $KL$ divergence estimator of $\mbox{\boldmath
$\theta$}$ as%
\[
\widehat{\mbox{\boldmath
$\theta$}}=\arg\underset{\mbox{\boldmath
$\theta$}\in\mathbf{\Theta}}{\inf}D\left(  f_{n}\left(  x\right)  ||f\left(
x;\mbox{\boldmath
$\theta$}\right)  \right)  .
\]

\cite{Lindsay:1994} proposed a version of \eqref{eq1.1} in discrete setting.
In recent years, many authors such as \cite{Morales:et:al:1995},
\cite{Jimenz:Shao:2001}, \cite{Broniatowski:Keziou:2009},
\cite{Broniatowski:2014}, \cite{Cherfi:2011,Cherfi:2012,Cherfi:2014} studied
the properties of minimum divergence estimators under different conditions.
\cite{Basu:et:al:2011} discussed in their book about the statistical inference
with the minimum distance approach.

Although the method of estimation based on $D\left(  f_{n}||f\right)  $ has
very interesting features, the definition is based on $f$ which, in general,
may not exist and also depends on $f_{n}$ which even if the number of samples
tends to infinity, there is no guarantee that converges to its true measure.

Let $X$ be a random variable with cumulative distribution function ($c.d.f$)
$F(x)$ and survival function ($s.f$) $\bar{F}\left(  x\right)  $. Based on $n$
observations $\{X_{1},\ldots,X_{n}\}$ of distribution $F\,$, define the
empirical cumulative distribution and survival functions, respectively, by%
\begin{equation}
F_{n}\left(  x\right)  =\sum_{i=1}^{n}\frac{i}{n}I_{\left[  X_{\left(
i\right)  },X_{\left(  i+1\right)  }\right)  }\left(  x\right)  ,
\label{eq1.2.1}%
\end{equation}
and%
\begin{equation}
\bar{F}_{n}\left(  x\right)  =\sum_{i=0}^{n-1}\left(  1-\frac{i}{n}\right)
I_{\left[  X_{\left(  i\right)  },X_{\left(  i+1\right)  }\right)  }\left(
x\right)  , \label{eq1.3}%
\end{equation}
where $I$ is the indicator function and $(0=X_{(0)}\leq)X_{(1)}\leq
X_{(2)}\leq\cdots\leq X_{(n)}$ are the ordered sample. $F_{n}(\bar{F}_{n})$ is
known in the literature as "empirical estimator" of $F(\bar{F})$.

In the case when $X$ and $Y$ are continuous nonnegative random variables with
$s.f$'s $\bar{F}$ and $\bar{G}$, respectively, a version of $KL$ in terms of
$s.f$'s $\bar{F}$ and $\bar{G}$ can be given as follows:%

\[
KLS\left(  \bar{F}|\bar{G}\right)  =\int_{0}^{\infty}\bar{F}\left(  x\right)
\log\frac{\bar{F}\left(  x\right)  }{\bar{G}(x)}dx-\left[  E\left(  X\right)
-E\left(  Y\right)  \right]  .
\]

The properties of this divergence measure are studied by some authors such as
\cite{Liu:2007} and \cite{Baratpour:HabibiRad:2012}.

In order to estimate the parameters of the model, \cite{Liu:2007} proposed
cumulative $KL$ divergence between the empirical survival function $\bar
{F}_{n}$ and survival function $\bar{F}$ (we call it $CKL\left(  \bar{F}%
_{n}||\bar{F}\right)  $) as
\begin{align*}
CKL\left(  \bar{F}_{n}||\bar{F}\right)   &  =\int\nolimits_{0}^{\infty}\bar
{F}_{n}\left(  x\right)  \log\frac{\bar{F}_{n}\left(  x\right)  }{\bar
{F}\left(  x;\mbox{\boldmath
$\theta$}\right)  }-\left[  \bar{F}_{n}\left(  x\right)  -\bar{F}\left(
x;\mbox{\boldmath
$\theta$}\right)  \right]  dx\\
&  =\int\nolimits_{0}^{\infty}\bar{F}_{n}\left(  x\right)  \log\bar{F}%
_{n}\left(  x\right)  dx-\int\nolimits_{0}^{\infty}\bar{F}_{n}\left(
x\right)  \log\bar{F}\left(  x;\mbox{\boldmath
$\theta$}\right)  dx-\left[  \bar{x}-E_{\mbox{\boldmath
$\theta$}}\left(  X\right)  \right]  .
\end{align*}

The cited author defined minimum $CKL$ divergence estimator ($MCKLE$) of
$\mbox{\boldmath
$\theta$}$ as%
\[
\widehat{\mbox{\boldmath
$\theta$}}=\arg\underset{\mbox{\boldmath
$\theta$}\in\mathbf{\Theta}}{\inf}CKL\left(  \bar{F}_{n}\left(  x\right)
||\bar{F}\left(  x;\mbox{\boldmath
$\theta$}\right)  \right)  .
\]

If consider the parts of $CKL\left(  \bar{F}_{n}||\bar{F}\right)  $ that
depends on $\mbox{\boldmath $\theta$}$ and define%
\begin{equation}
g\left(  \mbox{\boldmath
$\theta$}\right)  =E_{\mbox{\boldmath
$\theta$}}\left(  X\right)  -\int\nolimits_{0}^{\infty}\bar{F}_{n}\left(
x\right)  \log\bar{F}\left(  x;\mbox{\boldmath
$\theta$}\right)  dx, \label{eq1.5}%
\end{equation}
then the $MCKLE$ of $\mbox{\boldmath
$\theta$}$ can equivalently be defined by%
\[
\widehat{\mbox{\boldmath
$\theta$}}=\arg\underset{\mbox{\boldmath
$\theta$}\in\mathbf{\Theta}}{\inf}g\left(  \mbox{\boldmath
$\theta$}\right)  .
\]

Two important advantages of this estimator are that one does not need to have
the density function and for large values of $n$ the empirical estimator
$F_{n}$ tends to the distribution function $F$. \cite{Liu:2007} applied this
estimator in uniform and exponential models and \cite{Yari:Saghafi:2012} and
\cite{Yari:et:al:2013} used it for estimating parameters of Weibull
distribution; see also \cite{Park:et:al:2012} and \cite{Hwang:Park:2013}.
\cite{Yari:et:al:2013} found a simple form of \eqref{eq1.5} as%
\begin{equation}
g\left(  \mbox{\boldmath
$\theta$}\right)  =E_{\mbox{\boldmath
$\theta$}}\left(  X\right)  -\frac{1}{n}\sum_{i=1}^{n}h\left(  x_{i}\right)
=E_{\mbox{\boldmath
$\theta$}}\left(  X\right)  -\overline{h\left(  x\right)  }, \label{eq2.4}%
\end{equation}
where $\overline{h\left(  x\right)  }=\frac{1}{n}\sum_{i=1}^{n}h\left(
x_{i}\right)  $ for any function $h$ on $x$, and%
\begin{equation}
h\left(  x\right)  =\int\nolimits_{0}^{x}\log\bar{F}\left(
y;\mbox{\boldmath
$\theta$}\right)  dy. \label{eq2.5}%
\end{equation}

They also proved that%
\[
E\left(  h\left(  X\right)  \right)  =\int\nolimits_{0}^{\infty}\bar{F}\left(
x;\mbox{\boldmath
$\theta$}\right)  \log\bar{F}\left(  x;\mbox{\boldmath
$\theta$}\right)  dx,
\]
which shows that if $n$ tends to infinity, then $CKL\left(  \bar{F}_{n}%
||\bar{F}\right)  $ converges to zero.

The aim of the present paper is to investigate properties of $MCKLE$. The rest
of the paper is organized as follows: In section \ref{S3(Some Extensions)}, we
propose an extension of the $MCKLE$ in the case when the support of the
distribution is real line and provide some examples. In Section
\ref{S4(Asymptotic properties of estimators)}, we show that the proposed
estimator is in the class of generalized estimating equations ($GEE$).
Asymptotic properties of $MCKLE$ such as consistency, normality are
investigated in this section. In Section
\ref{S4(Asymptotic properties of estimators)}, we also provide some asymptotic
confidence intervals and asymptotic tests statistics based on $MCKLE$ to make
some inference on the parameters of the distribution.

\section{An Extension of $MCKLE$}

\label{S3(Some Extensions)} In this section, we propose an extension of the
$MCKLE$ for the case when $X$ is assumed to be a continuous random variable
with support $%
\mathbb{R}
$. It is known that%
\[
E_{\mbox{\boldmath
$\theta$}}\left\vert X\right\vert =\int\nolimits_{-\infty}^{0}F\left(
x\right)  dx+\int\nolimits_{0}^{\infty}\bar{F}\left(  x\right)  dx,
\]
\citep[see,][]{Rohatgi:Saleh:2015}.

We define the $CKL$ divergence and $CKL$ estimator in Liu approach as follows.

\begin{definition}
\label{def(5.1)}Let $X$ and $Y$ be random variables on $\mathbb{R}$ with
$c.d.f$'s \ $F\left(  x\right)  $ and $G\left(  x\right)  $, $s.f$'s $\bar
{F}\left(  x\right)  $ and $\bar{G}\left(  x\right)  $, finite means $E\left(
X\right)  $ and $E\left(  Y\right)  $, respectively. The $CKL$ divergence of
$\bar{F}$ relative to $\bar{G}$ is defined as%
\begin{align*}
CKL\left(  \bar{F}||\bar{G}\right)   &  =\int\nolimits_{-\infty}^{0}\left\{
F\left(  x\right)  \log\frac{F\left(  x\right)  }{G\left(  x\right)  }-\left[
F\left(  x\right)  -G\left(  x\right)  \right]  \right\}  dx\\
&  +\int\nolimits_{0}^{\infty}\left\{  \bar{F}\left(  x\right)  \log\frac
{\bar{F}\left(  x\right)  }{\bar{G}\left(  x\right)  }-\left[  \bar{F}\left(
x\right)  -\bar{G}\left(  x\right)  \right]  \right\}  dx\\
&  =\int\nolimits_{-\infty}^{0}F\left(  x\right)  \log\frac{F\left(  x\right)
}{G\left(  x\right)  }dx+\int\nolimits_{0}^{\infty}\bar{F}\left(  x\right)
\log\frac{\bar{F}\left(  x\right)  }{\bar{G}\left(  x\right)  }dx-\left[
E\left\vert X\right\vert -E\left\vert Y\right\vert \right]  .
\end{align*}

\end{definition}

An application of the log-sum inequality and the fact that $x\log\frac{x}%
{y}\geq x-y,\forall x,y>0$ (equality holds if and only if $x=y$) show that the
$CKL$ is non-negative. Using the fact that in log-sum inequality, equality
holds if and only if $F=G$, $a.s.$, one gets that $CKL\left(  \bar{F}||\bar
{G}\right)  =0$ if and only if $F=G$, $a.s.$ .

Let $F\left(  x;\mbox{\boldmath
$\theta$}\right)  $ be the population $c.d.f.$ with unknown parameters
$\mbox{\boldmath
$\theta$}$ and $F_{n}\left(  x\right)  $ be the empirical $c.d.f.$ based on a
random sample $X_{1},X_{2},\dots,X_{n}$ from $F\left(  x;\mbox{\boldmath
$\theta$}\right)  $. Based on above definition, the $CKL$ divergence of
$\bar{F}_{n}$ relative to $\bar{F}$ is defined as%
\[
CKL\left(  \bar{F}_{n}||\bar{F}\right)  =\int\nolimits_{-\infty}^{0}%
F_{n}\left(  x\right)  \log\frac{F_{n}\left(  x\right)  }{F\left(
x;\mbox{\boldmath
$\theta$}\right)  }dx+\int\nolimits_{0}^{\infty}\bar{F}_{n}\left(  x\right)
\log\frac{\bar{F}_{n}\left(  x\right)  }{\bar{F}\left(  x;\mbox{\boldmath
$\theta$}\right)  }dx-\left[  \bar{\left\vert x\right\vert }%
-E_{\mbox{\boldmath
$\theta$}}\left\vert X\right\vert \right]  ,
\]
where $\bar{|x|}$ is the mean of absolute values of the observations. Let us
also define%
\begin{equation}
g\left(  \mbox{\boldmath
$\theta$}\right)  =E_{\mbox{\boldmath
$\theta$}}\left\vert X\right\vert -\int\nolimits_{-\infty}^{0}F_{n}\left(
x\right)  \log F\left(  x;\mbox{\boldmath
$\theta$}\right)  dx-\int\nolimits_{0}^{\infty}\bar{F}_{n}\left(  x\right)
\log\bar{F}\left(  x;\mbox{\boldmath
$\theta$}\right)  dx. \label{eq5.13}%
\end{equation}

If $E_{\mbox{\boldmath
$\theta$}}\left\vert X\right\vert <\infty$ and $g^{\prime\prime}%
(\mbox{\boldmath
$\theta$})$ is positive definite, then we define $MCKLE$ of $\mbox{\boldmath
$\theta$}$ to be a value in the parameter space $\mathbf{\Theta}$ which
minimizes $g(\mbox{\boldmath
$\theta$})$. If $k=0$ (i.e., $X$ is nonnegative), then $g\left(
\mbox{\boldmath
$\theta$}\right)  $\ in \eqref{eq5.13}\ reduces to \eqref{eq1.5}.  So the
results of \cite{Liu:2007}, \cite{Yari:Saghafi:2012}, \cite{Yari:et:al:2013},
\cite{Park:et:al:2012} and \cite{Hwang:Park:2013} yield as special case.

It should be noted that by the law of large numbers $F_{n}\left(  x\right)  $
converges to $F\left(  x\right)  $ and $\bar{F}_{n}\left(  x\right)  $
converges to $\bar{F}\left(  x\right)  $ as $n$ tends to infinity.
Consequently $CKL\left(  \bar{F}_{n}||\bar{F}\right)  $ converges to zero. As
a consequence, if we take $\widehat{\mbox{\boldmath
$\theta$}}_{n}=T\left(  F_{n}\right)  $, then it is Fisher consistent, i.e.,
$T\left(  F\right)  =\mbox{\boldmath
$\theta$}$ (see, \cite{Fisher:1922} and \cite{Lindsay:1994}).

In order to study the properties of the estimator, we first find a simple form
of \eqref{eq5.13}. Let us introduce the following notations.%

\[
u\left(  x\right)  =\int\nolimits_{x}^{0}\log F\left(  y;\mbox{\boldmath
$\theta$}\right)  dy,
\]
and%
\begin{equation}
s\left(  x\right)  =I_{\left(  -\infty,0\right)  }\left(  x\right)  u\left(
x\right)  +I_{\left[  0,\infty\right)  }\left(  x\right)  h\left(  x\right)  ,
\label{eq5.8.1}%
\end{equation}
where $h$ is defined in \eqref{eq2.5}. Assuming that $x_{(1)},x_{(2)}%
,\dots,x_{(n)}$ denote the ordered observed values of the sample and that
$x_{\left(  k\right)  }<0\leq x_{\left(  k+1\right)  }$, for some value of
$k$, $k=0,\dots,n$. Then by \eqref{eq1.2.1}, \eqref{eq1.3}\ and
\eqref{eq5.13}, we have%
\begin{align*}
\int\nolimits_{-\infty}^{0}F_{n}\left(  x\right)  \log F\left(
x;\mbox{\boldmath
$\theta$}\right)  dx  &  =\int\nolimits_{-\infty}^{0}\sum_{i=1}^{n}\frac{i}%
{n}I_{\left[  x_{\left(  i\right)  },x_{\left(  i+1\right)  }\right)  }\left(
x\right)  \log F\left(  x;\mbox{\boldmath
$\theta$}\right)  dx\\
&  =\sum_{i=1}^{n}\frac{i}{n}\int\nolimits_{-\infty}^{\infty}I_{\left(
-\infty,0\right)  }\left(  x\right)  I_{\left[  x_{\left(  i\right)
},x_{\left(  i+1\right)  }\right)  }\left(  x\right)  \log F\left(
x;\mbox{\boldmath
$\theta$}\right)  dx\\
&  =\sum_{i=1}^{k-1}\frac{i}{n}\int\limits_{x_{\left(  i\right)  }%
}^{x_{\left(  i+1\right)  }}\log F\left(  x;\mbox{\boldmath
$\theta$}\right)  dx+\frac{k}{n}\int\limits_{x_{\left(  k\right)  }}^{0}\log
F\left(  x;\mbox{\boldmath
$\theta$}\right)  dx\\
&  =\frac{1}{n}\sum_{i=1}^{k-1}i\left[  u\left(  x_{\left(  i\right)
}\right)  -u\left(  x_{\left(  i+1\right)  }\right)  \right]  +\frac{k}%
{n}u\left(  x_{\left(  k\right)  }\right) \\
&  =\frac{1}{n}\sum_{i=1}^{k-1}\left[  i~u\left(  x_{\left(  i\right)
}\right)  -\left(  i+1\right)  u\left(  x_{\left(  i+1\right)  }\right)
\right]  +\frac{1}{n}\sum_{i=1}^{k-1}u\left(  x_{\left(  i+1\right)  }\right)
+\frac{k}{n}u\left(  x_{\left(  k\right)  }\right) \\
&  =\frac{1}{n}\left[  u\left(  x_{\left(  1\right)  }\right)  -k~u\left(
x_{\left(  k\right)  }\right)  \right]  +\frac{1}{n}\sum_{i=2}^{k}u\left(
x_{\left(  i\right)  }\right)  +\frac{k}{n}u\left(  x_{\left(  k\right)
}\right) \\
&  =\frac{1}{n}\sum_{i=1}^{k}u\left(  x_{\left(  i\right)  }\right)  .
\end{align*}

Using the same steps, we have%
\[
\int\nolimits_{0}^{\infty}\bar{F}_{n}\left(  x\right)  \log\bar{F}\left(
x;\mbox{\boldmath
$\theta$}\right)  dx=\frac{1}{n}\sum_{i=k+1}^{n}h\left(  x_{\left(  i\right)
}\right)  .
\]

So $g\left(  \mbox{\boldmath
$\theta$}\right)  $ in \eqref{eq5.13}\ gets the simple form%
\begin{align}
g\left(  \mbox{\boldmath
$\theta$}\right)   &  =E_{\mbox{\boldmath
$\theta$}}\left\vert X\right\vert -\frac{1}{n}\sum_{i=1}^{k}u\left(
x_{\left(  i\right)  }\right)  -\frac{1}{n}\sum_{i=k+1}^{n}h\left(  x_{\left(
i\right)  }\right) \nonumber\\
&  =E_{\mbox{\boldmath
$\theta$}}\left\vert X\right\vert -\frac{1}{n}\sum_{i=1}^{n}\left[  I_{\left(
-\infty,0\right)  }\left(  x_{i}\right)  u\left(  x_{i}\right)  +I_{\left[
0,\infty\right)  }\left(  x_{i}\right)  h\left(  x_{i}\right)  \right]
\nonumber\\
&  =E_{\mbox{\boldmath
$\theta$}}\left\vert X\right\vert -\frac{1}{n}\sum_{i=1}^{n}s\left(
x_{i}\right)  =E_{\mbox{\boldmath
$\theta$}}\left\vert X\right\vert -\overline{s\left(  x\right)  }.
\label{eq5.8.4}%
\end{align}

If $k=0$ (i.e., $X$ is nonnegative), then $g\left(  \mbox{\boldmath
$\theta$}\right)  $\ in \eqref{eq5.8.4}\ reduces to \eqref{eq2.4}. It can be
easily seen that%
\[
E\left(  s\left(  X\right)  \right)  =\int\nolimits_{-\infty}^{0}F\left(
x;\mbox{\boldmath
$\theta$}\right)  \log F\left(  x;\mbox{\boldmath
$\theta$}\right)  dx+\int\nolimits_{0}^{\infty}\bar{F}\left(
x;\mbox{\boldmath
$\theta$}\right)  \log\bar{F}\left(  x;\mbox{\boldmath
$\theta$}\right)  dx,
\]
which proves that if $n$ tends to infinity, then $CKL\left(  \bar{F}_{n}%
||\bar{F}\right)  $ converges to zero.

In The following, we give some examples.

\begin{example}
Let $\{X_{1},\ldots,X_{n}\}$ be sequence of $i.i.d.$ Normal random variables
with probability density function
\[
\phi\left(  x;\mu,\sigma\right)  =\frac{1}{\sqrt{2\pi\sigma^{2}}}\exp\left(
-\frac{1}{2}\left(  \frac{x-\mu}{\sigma}\right)  ^{2}\right)  ,\text{
\ \ \ }x\in%
\mathbb{R}
.
\]

In this case $E\left(  \left\vert X\right\vert \right)  =\mu\left[
2\Phi\left(  \frac{\mu}{\sigma}\right)  -1\right]  +2\sigma\phi\left(
\frac{\mu}{\sigma}\right)  $, where $\Phi$ denotes the distribution function
of standard normal. For this distribution, $h\left(  x\right)  $, $u\left(
x\right)  $ and $g\left(  \mu,\sigma\right)  $ do not have close forms. The
derivative of $g\left(  \mu,\sigma\right)  $ with respect to $\mu$ and
$\sigma$ and setting the results to zero gives respectively%
\[
2n\Phi\left(  \frac{\mu}{\sigma}\right)  -n-\sum_{\substack{i=1\\x_{i}<0}%
}^{k}\log\Phi\left(  \frac{x_{i}-\mu}{\sigma}\right)  +k\log\Phi\left(
-\frac{\mu}{\sigma}\right)  +\sum_{\substack{i=k+1\\x_{i}\geq0}}^{n}\log
\Phi\left(  \frac{\mu-x_{i}}{\sigma}\right)  -\left(  n-k\right)  \log
\Phi\left(  \frac{\mu}{\sigma}\right)  =0,
\]
and%
\begin{equation}
2n\phi\left(  \frac{\mu}{\sigma}\right)  +\sum_{\substack{i=1\\x_{i}<0}%
}^{k}\int\nolimits_{\frac{x_{i}-\mu}{\sigma}}^{-\frac{\mu}{\sigma}}\frac
{z\phi\left(  z\right)  }{\Phi\left(  z\right)  }dz-\sum
_{\substack{i=k+1\\x_{i}\geq0}}^{n}\int\nolimits_{-\frac{\mu}{\sigma}}%
^{\frac{x_{i}-\mu}{\sigma}}\frac{z\phi\left(  z\right)  }{1-\Phi\left(
z\right)  }dz=0. \label{eq5.20}%
\end{equation}

To obtain our estimators, we need to solve these equations which should be
solved numerically. For computational purposes, the following equivalent
equation can be solved instead of \eqref{eq5.20}.%
\[
2\phi\left(  \frac{\mu}{\sigma}\right)  +\int\nolimits_{\frac{x_{\left(
1\right)  }-\mu}{\sigma}}^{-\frac{\mu}{\sigma}}F_{n}\left(  \mu+\sigma
z\right)  \frac{z\phi\left(  z\right)  }{\Phi\left(  z\right)  }%
dz-\int\nolimits_{-\frac{\mu}{\sigma}}^{\frac{x_{\left(  n\right)  }-\mu
}{\sigma}}\bar{F}_{n}\left(  \mu+\sigma z\right)  \frac{z\phi\left(  z\right)
}{1-\Phi\left(  z\right)  }dz=0.
\]

Figure \ref{figgNormal} represents $g\left(  \mu,\sigma\right)  $ for a
simulated sample of size $100$ from Normal distribution with parameters
$\left(  \mu=2,\sigma=3\right)  $. {This figure shows that in this case
$g(\mu,\sigma)$ has minimum and hence the estimators of $\mu$ and $\sigma$ are
the values that minimize $g\left(  \mu,\sigma\right)  $.}

\begin{figure}[ptb]
\begin{center}
{\includegraphics[height=4in] {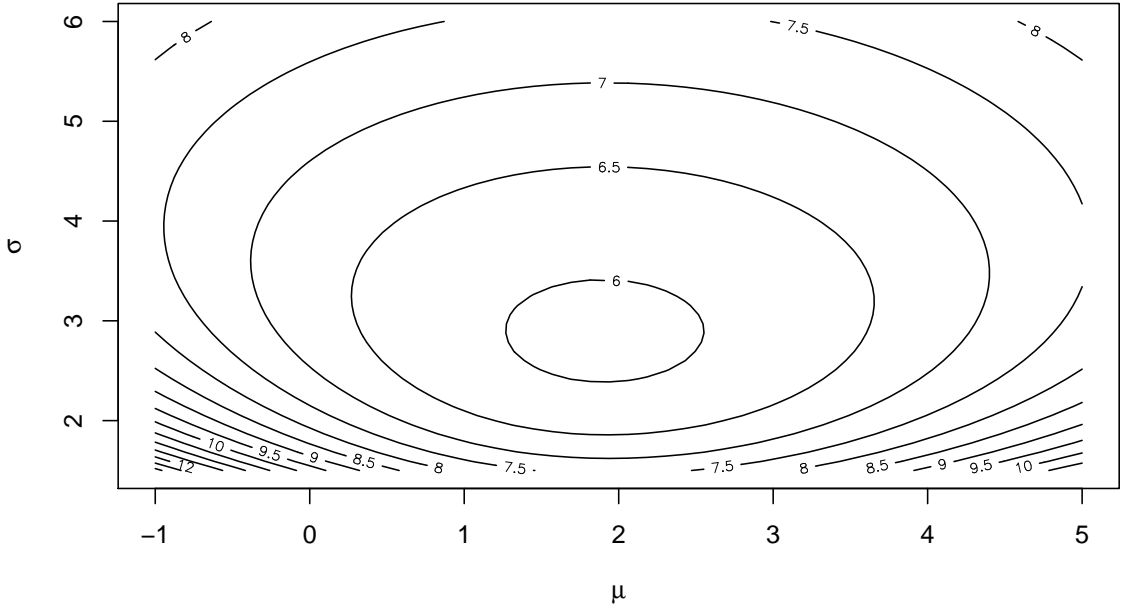}\\[0pt]}
\end{center}
\caption{$g\left(  \mu,\sigma\right)  $ for a simulated sample of size $100$
from Normal distribution with parameters $\left(  \mu=2,\sigma=3\right)  $}%
\label{figgNormal}%
\end{figure}

Figure \ref{figNormal} compares these estimators with the corresponding
$MLE$'s. In order to compare our estimators and the $MLE$'s we made a
simulation study in which we used samples of sizes $10$ to $55$ by $5$ with
$10000$ repeats, where we assume that the true values of the model parameters
are $\mu_{true}=2$ and $\sigma_{true}=3$. It is evident from the plots that
the $MCKLE$ approximately coincides with the $MLE$ in both cases.

\begin{figure}[ptb]
\begin{center}
{\includegraphics[height=4in] {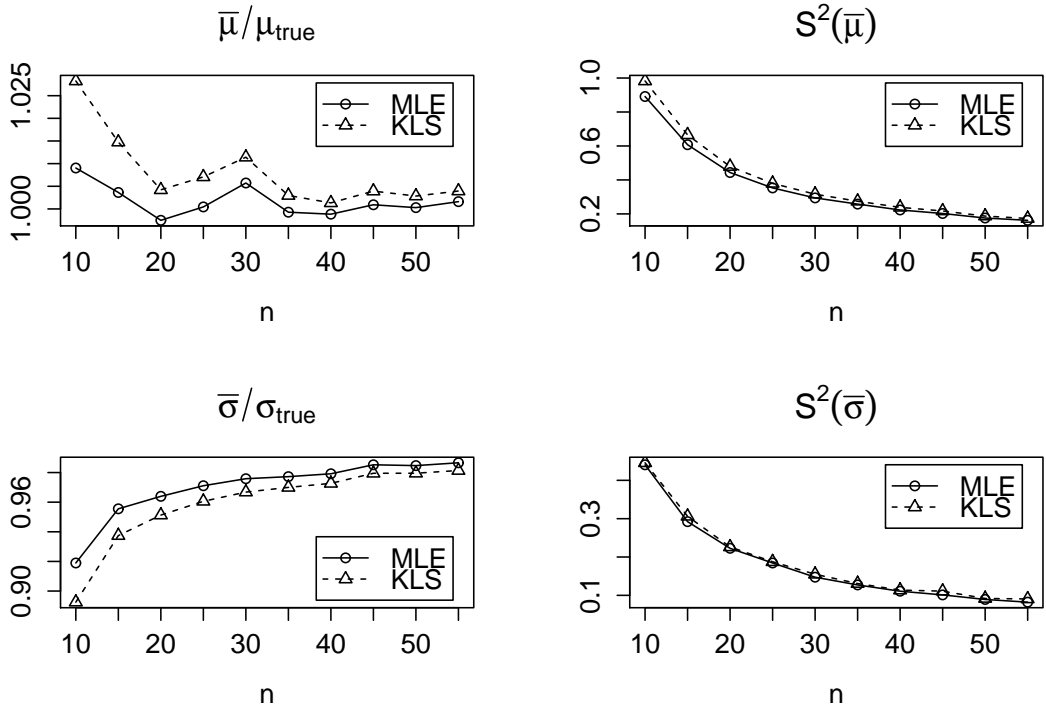}\\[0pt]}
\end{center}
\caption{$\bar{\mu}/\mu_{true}$, $S^{2}\left(  \bar{\mu}\right)  $,
$\bar{\sigma}/\sigma_{true}$ and $S^{2}\left(  \bar{\sigma}\right)  $ as
functions of sample size}%
\label{figNormal}%
\end{figure}
\end{example}

\begin{example}
\label{ex(5.2)}Let $\{X_{1},\ldots,X_{n}\}$ be sequence of $i.i.d.$ Laplace
random variables with probability density function
\[
f\left(  x;\theta\right)  =\frac{1}{\theta}\exp\left\vert \frac{x}{\theta
}\right\vert ,\text{ \ \ \ }x\in%
\mathbb{R}
.
\]

We simply have $MCKLE$ of $\theta$ as%
\[
\widehat{\theta}=\sqrt{\frac{\bar{X^{2}}}{2}}.
\]

For asymptotic properties of this estimator see Section
\ref{S4(Asymptotic properties of estimators)}.
\end{example}

\section{Asymptotic properties of estimators}

\label{S4(Asymptotic properties of estimators)}

In this section we study asymptotic properties of $MCKLE$'s. For this purpose,
first we give a brief review on $GEE$. Some related references on $GEE$ are
\cite{Huber:1964}, \citet[chapter 7]{Serfling:1980}, \cite{Qin:Lawless:1994},
\citet[chapter 5]{vanderVaart:2000}, \citet[chapter 14]{Pawitan:2001},
\citet[chapter 5]{Shao:2003}, \citet[chapter 3]{Huber:Ronchetti:2009} and
\cite{Hampel:et:al:2011}.

Throughout this section, we use the terminology used by \cite{Shao:2003}. We
assume that $X_{1},...,X_{n}$ represents independent (not necessarily
identically distributed) random vectors, in which the dimension of $X_{i}$ is
$d_{i},~i=1,...,n~\left(  \sup_{i}d_{i}<\infty\right)  $. We also assume that
in the population model the vector $\mbox{\boldmath
$\theta$}$ is a $k$-vector of unknown parameters. The $GEE$ method is a
general method in statistical inference for deriving point estimators. Let
$\mbox{\boldmath
$\Theta$}\subset%
\mathbb{R}
^{k}$ be the range of $\mbox{\boldmath
$\theta$}$, $\mbox{\boldmath
$\psi$}_{i}$ be a Borel function from $%
\mathbb{R}
^{d_{i}}\times\mbox{\boldmath
$\Theta$}$ to $%
\mathbb{R}
^{k}$, $i=1,...,n$, and%
\[
s_{n}(\mbox{\boldmath
$\gamma$})=\sum_{i=1}^{n}\mbox{\boldmath
$\psi$}_{i}\left(  X_{i},\mbox{\boldmath
$\gamma$}\right)  ,~\mbox{\boldmath
$\gamma$}\in\mbox{\boldmath
$\Theta$}.
\]

If $\widehat{\mbox{\boldmath
$\theta$}}\in\mbox{\boldmath
$\Theta$}$ is an estimator of $\mbox{\boldmath
$\theta$}$ which satisfies $s_{n}(\widehat{\mbox{\boldmath
$\theta$}})=0$, then $\widehat{\mbox{\boldmath
$\theta$}}$ is called a $GEE$ estimator. The equation $s_{n}\left(
\mbox{\boldmath
$\gamma$}\right)  =0$ is called a $GEE$. Most of the estimation methods such
as likelihood estimators, moment estimators and M-estimators are special cases
of $GEE$ estimators. Usually $GEE$'s are chosen such that%
\begin{equation}
E\left[  s_{n}\left(  \mbox{\boldmath
$\theta$}\right)  \right]  =\sum_{i=1}^{n}E\left[  \mbox{\boldmath
$\psi$}_{i}\left(  X_{i},\mbox{\boldmath
$\theta$}\right)  \right]  =0. \label{eq6.2}%
\end{equation}

If the exact expectation does not exist, then the expectation $E$ may be
replaced by an asymptotic expectation. The consistency and asymptotic
normality of the $GEE$ are studied by the authors under different conditions \citep[see, fore example][]{Shao:2003}.

\subsection{Consistency and asymptotic normality of the $MCKLE$}

Let $\widehat{\mbox{\boldmath $\theta$}}_{n}$ be $MCKLE$ by minimizing
$g\left(  \mbox{\boldmath
$\theta$}\right)  $ in \eqref{eq5.8.4} with $s\left(  x\right)  $ as defined
in \eqref{eq5.8.1}. Here, we show that the $MCKLE$'s are special cases of
$GEE$. Using this, we show consistency and asymptotic normality of $MCKLE$'s.

\begin{theorem}
\label{th(7.2)}$MCKLE$'s, by minimizing $g\left(  \mbox{\boldmath
$\theta$}\right)  $ in \eqref{eq5.8.4}, are special cases of $GEE$ estimators.
\end{theorem}

\begin{proof} In order to
minimize $g\left( \mbox{\boldmath  $\theta$}\right) $ in \eqref{eq5.8.4}, we
get the derivative of $g\left( \mbox{\boldmath  $\theta$}\right) $, under the assumption that it exists,
\begin{equation*}
\frac{\partial }{\partial \mbox{\boldmath
$\theta$}}g\left( \mbox{\boldmath
$\theta$}\right) =\frac{\partial }{\partial \mbox{\boldmath
$\theta$}}E_{\mbox{\boldmath
$\theta$}}\left\vert X\right\vert -\frac{1}{n}\sum_{i=1}^{n}\frac{\partial }{%
\partial \mbox{\boldmath
$\theta$}}s\left( x_{i}\right) =0,
\end{equation*}%
which is equivalent to $GEE$ $s_{n}\left( \mbox{\boldmath
$\theta$}\right) =0$ where%
\begin{equation}
s_{n}\left( \mbox{\boldmath
$\theta$}\right) =\sum_{i=1}^{n}\left[ \frac{\partial }{\partial
\mbox{\boldmath
$\theta$}}E_{\mbox{\boldmath
$\theta$}}\left\vert X\right\vert -\frac{\partial }{\partial
\mbox{\boldmath
$\theta$}}s\left( x_{i}\right) \right] =\sum_{i=1}^{n}%
\mbox{\boldmath
$\psi$}\left( x_{i},\mbox{\boldmath
$\theta$}\right) ,  \label{eq7.9}
\end{equation}%
with%
\begin{equation}
\mbox{\boldmath
$\psi$}\left( x,\mbox{\boldmath
$\theta$}\right) =\frac{\partial }{\partial \mbox{\boldmath
$\theta$}}E_{\mbox{\boldmath
$\theta$}}\left\vert X\right\vert -\frac{\partial }{\partial
\mbox{\boldmath
$\theta$}}s\left( x\right) .  \label{eq7.9.1}
\end{equation}
We must prove that $E\left[ s_{n}\left( \mbox{\boldmath
$\theta$}\right) \right] =0$ or equivalently $E\left[
\mbox{\boldmath
$\psi$}\left( X,\mbox{\boldmath
$\theta$}\right) \right] =0$. We have
\begin{equation*}
E\left[ \mbox{\boldmath
$\psi$}\left( X,\mbox{\boldmath
$\theta$}\right) \right] =\frac{\partial }{\partial
\mbox{\boldmath
$\theta$}}E_{\mbox{\boldmath
$\theta$}}\left\vert X\right\vert -E\left[ \frac{\partial }{\partial
\mbox{\boldmath
$\theta$}}s\left( X\right) \right] .
\end{equation*}
So, it is enough to show that%
\begin{equation*}
E\left[ \frac{\partial }{\partial \mbox{\boldmath
$\theta$}}s\left( X\right) \right] =\frac{\partial }{\partial
\mbox{\boldmath
$\theta$}}E_{\mbox{\boldmath
$\theta$}}\left\vert X\right\vert .
\end{equation*}%
By simple algebra we have%
\begin{eqnarray}
E\left[ \frac{\partial }{\partial \mbox{\boldmath
$\theta$}}s\left( X\right) \right]  &=&\int\nolimits_{-\infty }^{0}\frac{%
\partial }{\partial \mbox{\boldmath
$\theta$}}u\left( x\right) f\left( x;\mbox{\boldmath
$\theta$}\right) dx+\int\nolimits_{0}^{\infty }\frac{\partial }{\partial
\mbox{\boldmath
$\theta$}}h\left( x\right) f\left( x;\mbox{\boldmath
$\theta$}\right) dx  \notag \\
&=&\frac{\partial }{\partial \mbox{\boldmath
$\theta$}}\left\{ \int\nolimits_{-\infty }^{0}F\left( y;%
\mbox{\boldmath
$\theta$}\right) dy+\int\nolimits_{0}^{\infty }\bar{F}\left( y;%
\mbox{\boldmath
$\theta$}\right) dy\right\}   \notag \\
&=&\frac{\partial }{\partial \mbox{\boldmath
$\theta$}}E_{\mbox{\boldmath
$\theta$}}\left\vert X\right\vert,  \label{eq7.1}
\end{eqnarray}%
which proves the result.
\end{proof}

\begin{corollary}
In special case that support of $X$ is $%
\mathbb{R}
^{+}$, $MCKLE$ is an special case of $GEE$ estimators, where%
\begin{equation}
s_{n}\left(  \mbox{\boldmath
$\theta$}\right)  =\sum_{i=1}^{n}\left[  \frac{\partial}{\partial
\mbox{\boldmath
$\theta$}}E_{\mbox{\boldmath
$\theta$}}\left(  X\right)  -\frac{\partial}{\partial\mbox{\boldmath
$\theta$}}h\left(  x_{i}\right)  \right]  =\sum_{i=1}^{n}\mbox{\boldmath
$\psi$}\left(  x_{i},\mbox{\boldmath
$\theta$}\right)  , \label{eq7.2}%
\end{equation}
with%
\begin{equation}
\mbox{\boldmath
$\psi$}\left(  x,\mbox{\boldmath
$\theta$}\right)  =\frac{\partial}{\partial\mbox{\boldmath
$\theta$}}E_{\mbox{\boldmath
$\theta$}}\left(  X\right)  -\frac{\partial}{\partial\mbox{\boldmath
$\theta$}}h\left(  x\right)  . \label{eq7.3}%
\end{equation}

\end{corollary}

We now study other conditions under which $MCKLE$'s are consistent. For each
$n$, let $\widehat{\mbox{\boldmath
$\theta$}}_{n}$ be an $MCKLE$ or equivalently a $GEE$ estimator, i.e.,
$s_{n}\left(  \widehat{\mbox{\boldmath
$\theta$}}_{n}\right)  =0$, where $s_{n}\left(  \mbox{\boldmath
$\theta$}\right)  $ is defined as \eqref{eq7.9} or \eqref{eq7.2}.
In the next Theorem, we study the regular consistency of $\widehat
{\mbox{\boldmath
$\theta$}}_{n}$.

\begin{theorem}
\label{th(7.3)}For each $n$, let $\widehat{\mbox{\boldmath
$\theta$}}_{n}$ be an $MCKLE$ or equivalently a $GEE$ estimator. Suppose that
$\mbox{\boldmath
$\psi$}$ which is defined in \eqref{eq7.9.1} or \eqref{eq7.3} is a bounded and
continuous function of $\mbox{\boldmath
$\theta$}$. Let%
\[
\mbox{\boldmath
$\Psi$}\left(  \mbox{\boldmath
$\theta$}\right)  =E\left[  \mbox{\boldmath
$\psi$}\left(  X,\mbox{\boldmath
$\theta$}\right)  \right]  ,
\]
where we assume that $\mbox{\boldmath
$\Psi$}^{\prime}\left(  \mbox{\boldmath
$\theta$}\right)  $ exists and is full rank. Then $\widehat{\mbox{\boldmath
$\theta$}}_{n}\overset{p}{\rightarrow}\mbox{\boldmath
$\theta$}$.
\end{theorem}

\begin{proof}The result follows from  Proposition 5.2 of \cite{Shao:2003}
using the fact that \eqref{eq6.2} holds.\end{proof}

Asymptotic normality of a consistent sequence of $MCKLE$'s can be established
under some conditions. We first consider the special case where
$\mbox{\boldmath
$\theta$}$ is scalar and $X_{1},...,X_{n}$ are $i.i.d.$ .

\begin{theorem}
\label{th(7.4)}For each $n$, let $\widehat{\mbox{\boldmath
$\theta$}}_{n}$ be an $MCKLE$ or equivalently a $GEE$ estimator. Then
\[
\sqrt{n}\left(  \widehat{\mbox{\boldmath
$\theta$}}_{n}-\mbox{\boldmath
$\theta$}\right)  \overset{d}{\rightarrow}N\left(  0,\sigma_{F}^{2}\right)  ,
\]
where $\sigma_{F}^{2}=A/B^{2}$, with%
\[
A=E\left[  \frac{\partial}{\partial\mbox{\boldmath
$\theta$}}s\left(  X\right)  \right]  ^{2}-\left[  \frac{\partial}%
{\partial\mbox{\boldmath
$\theta$}}E_{\mbox{\boldmath
$\theta$}}\left\vert X\right\vert \right]  ^{2},
\]
and
\[
B=\int\nolimits_{-\infty}^{0}\frac{\left[  \frac{\partial}{\partial
\mbox{\boldmath
$\theta$}}F\left(  x;\mbox{\boldmath
$\theta$}\right)  \right]  ^{2}}{F\left(  x;\mbox{\boldmath
$\theta$}\right)  }dx+\int\nolimits_{0}^{\infty}\frac{\left[  \frac{\partial
}{\partial\mbox{\boldmath
$\theta$}}\bar{F}\left(  x;\mbox{\boldmath
$\theta$}\right)  \right]  ^{2}}{\bar{F}\left(  x;\mbox{\boldmath
$\theta$}\right)  }dx.
\]

\end{theorem}

\begin{proof}
Using Theorem \ref{th(7.2)} we have $E\left[  \mbox{\boldmath
$\psi$}\left(  X,\mbox{\boldmath
$\theta$}\right)  \right]  =0$. So if consider $\mbox{\boldmath
$\psi$}$ defined in \eqref{eq7.9.1}
\begin{align*}
E\left[  \mbox{\boldmath
$\psi$}\left(  X,\mbox{\boldmath
$\theta$}\right)  \right]  ^{2}  & =Var\left[  \mbox{\boldmath
$\psi$}\left(  X,\mbox{\boldmath
$\theta$}\right)  \right]  \\
& =Var\left[  \frac{\partial}{\partial\mbox{\boldmath
$\theta$}}E_{\mbox{\boldmath
$\theta$}}\left\vert X\right\vert -\frac{\partial}{\partial\mbox{\boldmath
$\theta$}}s\left(  X\right)  \right]  \\
& =Var\left[  \frac{\partial}{\partial\mbox{\boldmath
$\theta$}}s\left(  X\right)  \right]  \\
& =E\left[  \frac{\partial}{\partial\mbox{\boldmath
$\theta$}}s\left(  X\right)  \right]  ^{2}-E^{2}\left[  \frac{\partial
}{\partial\mbox{\boldmath
$\theta$}}s\left(  X\right)  \right]  \\
& =E\left[  \frac{\partial}{\partial\mbox{\boldmath
$\theta$}}s\left(  X\right)  \right]  ^{2}-\left[  \frac{\partial}%
{\partial\mbox{\boldmath
$\theta$}}E_{\mbox{\boldmath
$\theta$}}\left\vert X\right\vert \right]  ^{2},
\end{align*}
where the last equality follows from \eqref{eq7.1}. On the other hand%
\[
\mbox{\boldmath
$\Psi$}^{\prime}\left(  \mbox{\boldmath
$\theta$}\right)  =\frac{\partial^{2}}{\partial\mbox{\boldmath
$\theta$}^{2}}E_{\mbox{\boldmath
$\theta$}}\left\vert X\right\vert -E\left[  \frac{\partial^{2}}{\partial
\mbox{\boldmath
$\theta$}^{2}}s\left(  X\right)  \right]  ,
\]
and%
\begin{align*}
E\left[  \frac{\partial^{2}}{\partial\mbox{\boldmath
$\theta$}^{2}}s\left(  X\right)  \right]    & =\int_{-\infty}^{0}%
\int\nolimits_{x}^{0}\frac{\partial^{2}}{\partial\mbox{\boldmath
$\theta$}^{2}}\log F\left(  y;\mbox{\boldmath
$\theta$}\right)  dyf\left(  x;\mbox{\boldmath
$\theta$}\right)  dx+\int\nolimits_{0}^{\infty}\int\nolimits_{0}^{x}%
\frac{\partial^{2}}{\partial\mbox{\boldmath
$\theta$}^{2}}\log\bar{F}\left(  y;\mbox{\boldmath
$\theta$}\right)  dyf\left(  x;\mbox{\boldmath
$\theta$}\right)  dx\\
& =\int_{-\infty}^{0}\left\{  \frac{\frac{\partial^{2}}{\partial
\mbox{\boldmath
$\theta$}^{2}}F\left(  y;\mbox{\boldmath
$\theta$}\right)  }{F\left(  y;\mbox{\boldmath
$\theta$}\right)  }-\left[  \frac{\frac{\partial}{\partial\mbox{\boldmath
$\theta$}}F\left(  y;\mbox{\boldmath
$\theta$}\right)  }{F\left(  y;\mbox{\boldmath
$\theta$}\right)  }\right]  ^{2}\right\}  F\left(  y;\mbox{\boldmath
$\theta$}\right)  dy\\
& +\int\nolimits_{0}^{\infty}\left\{  \frac{\frac{\partial^{2}}{\partial
\mbox{\boldmath
$\theta$}^{2}}\bar{F}\left(  y;\mbox{\boldmath
$\theta$}\right)  }{\bar{F}\left(  y;\mbox{\boldmath
$\theta$}\right)  }-\left[  \frac{\frac{\partial}{\partial\mbox{\boldmath
$\theta$}}\bar{F}\left(  y;\mbox{\boldmath
$\theta$}\right)  }{\bar{F}\left(  y;\mbox{\boldmath
$\theta$}\right)  }\right]  ^{2}\right\}  \bar{F}\left(  y;\mbox{\boldmath
$\theta$}\right)  dy\\
& =\frac{\partial^{2}}{\partial\mbox{\boldmath
$\theta$}^{2}}\int_{-\infty}^{0}F\left(  x;\mbox{\boldmath
$\theta$}\right)  dx-\int_{-\infty}^{0}\frac{\left[  \frac{\partial}%
{\partial\mbox{\boldmath
$\theta$}}F\left(  x;\mbox{\boldmath
$\theta$}\right)  \right]  ^{2}}{F\left(  x;\mbox{\boldmath
$\theta$}\right)  }dx\\
& +\frac{\partial^{2}}{\partial\mbox{\boldmath
$\theta$}^{2}}\int\nolimits_{0}^{\infty}\bar{F}\left(  x;\mbox{\boldmath
$\theta$}\right)  dx-\int\nolimits_{0}^{\infty}\frac{\left[  \frac{\partial
}{\partial\mbox{\boldmath
$\theta$}}\bar{F}\left(  x;\mbox{\boldmath
$\theta$}\right)  \right]  ^{2}}{\bar{F}\left(  x;\mbox{\boldmath
$\theta$}\right)  }dx\\
& =\frac{\partial^{2}}{\partial\mbox{\boldmath
$\theta$}^{2}}E_{\mbox{\boldmath
$\theta$}}\left\vert X\right\vert -\int_{-\infty}^{0}\frac{\left[
\frac{\partial}{\partial\mbox{\boldmath
$\theta$}}F\left(  x;\mbox{\boldmath
$\theta$}\right)  \right]  ^{2}}{F\left(  x;\mbox{\boldmath
$\theta$}\right)  }dx-\int\nolimits_{0}^{\infty}\frac{\left[  \frac{\partial
}{\partial\mbox{\boldmath
$\theta$}}\bar{F}\left(  x;\mbox{\boldmath
$\theta$}\right)  \right]  ^{2}}{\bar{F}\left(  x;\mbox{\boldmath
$\theta$}\right)  }dx.
\end{align*}
So%
\[
\mbox{\boldmath
$\Psi$}^{\prime}\left(  \mbox{\boldmath
$\theta$}\right)  =\int_{-\infty}^{0}\frac{\left[  \frac{\partial}%
{\partial\mbox{\boldmath
$\theta$}}F\left(  x;\mbox{\boldmath
$\theta$}\right)  \right]  ^{2}}{F\left(  x;\mbox{\boldmath
$\theta$}\right)  }dx+\int\nolimits_{0}^{\infty}\frac{\left[  \frac{\partial
}{\partial\mbox{\boldmath
$\theta$}}\bar{F}\left(  x;\mbox{\boldmath
$\theta$}\right)  \right]  ^{2}}{\bar{F}\left(  x;\mbox{\boldmath
$\theta$}\right)  }dx.
\]
Now, using Theorem 5.13 of \cite{Shao:2003}, $\sigma_{F}^{2}$ will be found.
\end{proof}

The next Theorem shows asymptotic normality of $MCKLE$'s, when
$\mbox{\boldmath
$\theta$}$ is vector and $X_{1},...,X_{n}$ are $i.i.d.$ .

\begin{theorem}
\label{th(7.5)}Under the conditions of Theorem 5.14 of \cite{Shao:2003},%
\[
V_{n}^{-1/2}\left(  \widehat{\mbox{\boldmath
$\theta$}}_{n}-\mbox{\boldmath
$\theta$}\right)  \overset{d}{\rightarrow}N_{k}\left(  0,I_{k}\right)  ,
\]
where $V_{n}=\frac{1}{n}B^{-1}AB^{-1}$ with%
\[
A=\left[  \frac{\partial}{\partial\mbox{\boldmath
$\theta$}}s\left(  X\right)  \right]  \left[  \frac{\partial}{\partial
\mbox{\boldmath
$\theta$}}s\left(  X\right)  \right]  ^{T}-\left[  \frac{\partial}%
{\partial\mbox{\boldmath
$\theta$}}E_{\mbox{\boldmath
$\theta$}}\left\vert X\right\vert \right]  \left[  \frac{\partial}%
{\partial\mbox{\boldmath
$\theta$}}E_{\mbox{\boldmath
$\theta$}}\left\vert X\right\vert \right]  ^{T},
\]
and%
\[
B=\int\nolimits_{-\infty}^{0}\frac{\left[  \frac{\partial}{\partial
\mbox{\boldmath
$\theta$}}F\left(  x;\mbox{\boldmath
$\theta$}\right)  \right]  \left[  \frac{\partial}{\partial\mbox{\boldmath
$\theta$}}F\left(  x;\mbox{\boldmath
$\theta$}\right)  \right]  ^{T}}{F\left(  x;\mbox{\boldmath
$\theta$}\right)  }dx+\int\nolimits_{0}^{\infty}\frac{\left[  \frac{\partial
}{\partial\mbox{\boldmath
$\theta$}}\bar{F}\left(  x;\mbox{\boldmath
$\theta$}\right)  \right]  \left[  \frac{\partial}{\partial\mbox{\boldmath
$\theta$}}\bar{F}\left(  x;\mbox{\boldmath
$\theta$}\right)  \right]  ^{T}}{\bar{F}\left(  x;\mbox{\boldmath
$\theta$}\right)  }dx,
\]
provided that $B$ is invertible matrix.
\end{theorem}

\begin{proof}The proof is similar to that of Theorem \ref{th(7.4)}.%
\end{proof}

\begin{remark}
In Theorems \ref{th(7.4)} and \ref{th(7.5)}, for special case that support of
$X$ is $%
\mathbb{R}
^{+}$, $A$ and $B$ are given, respectively, by%
\[
A=E\left[  \frac{\partial}{\partial\mbox{\boldmath
$\theta$}}h\left(  X\right)  \right]  ^{2}-\left[  \frac{\partial}%
{\partial\mbox{\boldmath
$\theta$}}E_{\mbox{\boldmath
$\theta$}}\left(  X\right)  \right]  ^{2},
\]%
\[
B=\int\nolimits_{0}^{\infty}\frac{\left[  \frac{\partial}{\partial
\mbox{\boldmath
$\theta$}}\bar{F}\left(  x;\mbox{\boldmath
$\theta$}\right)  \right]  ^{2}}{\bar{F}\left(  x;\mbox{\boldmath
$\theta$}\right)  }dx,
\]
and%
\[
A=E\left[  \frac{\partial}{\partial\mbox{\boldmath
$\theta$}}h\left(  X\right)  \right]  \left[  \frac{\partial}{\partial
\mbox{\boldmath
$\theta$}}h\left(  X\right)  \right]  ^{T}-\left[  \frac{\partial}%
{\partial\mbox{\boldmath
$\theta$}}E_{\mbox{\boldmath
$\theta$}}\left(  X\right)  \right]  \left[  \frac{\partial}{\partial
\mbox{\boldmath
$\theta$}}E_{\mbox{\boldmath
$\theta$}}\left(  X\right)  \right]  ^{T},
\]%
\[
B=\int\nolimits_{0}^{\infty}\frac{\left[  \frac{\partial}{\partial
\mbox{\boldmath
$\theta$}}\bar{F}\left(  x;\mbox{\boldmath
$\theta$}\right)  \right]  \left[  \frac{\partial}{\partial\mbox{\boldmath
$\theta$}}\bar{F}\left(  x;\mbox{\boldmath
$\theta$}\right)  \right]  ^{T}}{\bar{F}\left(  x;\mbox{\boldmath
$\theta$}\right)  }dx.
\]

\end{remark}

Now, following \cite{Pawitan:2001}, we can find sample version of the variance
formula for the $MCKLE$ as follows. Given $x_{1},...,x_{n}$ let%
\begin{equation}
J=\widehat{E}\left[  \mbox{\boldmath
$\psi$}\left(  X,\mbox{\boldmath
$\theta$}\right)  \right]  ^{2}=\frac{1}{n}\sum_{i=1}^{n}\mbox{\boldmath
$\psi$}^{2}\left(  x_{i},\widehat{\mbox{\boldmath
$\theta$}}\right)  , \label{eq7.37}%
\end{equation}
where in the vector case we would simply use $\mbox{\boldmath
$\psi$}\left(  x_{i},\widehat{\mbox{\boldmath
$\theta$}}\right)  \mbox{\boldmath
$\psi$}^{T}\left(  x_{i},\widehat{\mbox{\boldmath
$\theta$}}\right)  $ in the summation, and%
\begin{equation}
I=-\widehat{E}\frac{\partial}{\partial\mbox{\boldmath
$\theta$}}\mbox{\boldmath
$\psi$}\left(  X,\mbox{\boldmath
$\theta$}\right)  =-\frac{1}{n}\sum_{i=1}^{n}\frac{\partial}{\partial
\mbox{\boldmath
$\theta$}}\mbox{\boldmath
$\psi$}\left(  x_{i},\widehat{\mbox{\boldmath
$\theta$}}\right)  . \label{eq7.38}%
\end{equation}

Then, we have the following result.

\begin{theorem}
\label{th(7.6)}Using notations defined in \eqref{eq7.37} and \eqref{eq7.38},
and an application of Slutsky's Theorem yields
\[
\widehat{V}_{n}^{-1/2}\left(  \widehat{\mbox{\boldmath
$\theta$}}_{n}-\mbox{\boldmath
$\theta$}\right)  \overset{d}{\rightarrow}N_{k}\left(  0,I_{k}\right)  ,
\]
where%
\begin{equation}
\widehat{V}_{n}=\frac{1}{n}I^{-1}JI^{-1}, \label{eq7.40.1}%
\end{equation}
provided that $I$ is invertible matrix, or equivalently $g\left(
\mbox{\boldmath
$\theta$}\right)  $ has infimum value on parameter space $\Theta$.
\end{theorem}

In Theorems \ref{th(7.4)} and \ref{th(7.5)}, the estimator $\widehat{V}_{n}$
is a sample version of $V_{n}$, see also \cite{Basu:Lindsay:1994}. It is also
known that the sample variance \eqref{eq7.40.1} is a robust estimation which
is known as the 'sandwich' estimator, with $I^{-1}$ as the bread and $J$ the
filling \citep[see,][]{Huber:1967}. In likelihood approach, the quantity $I$
is the usual observed Fisher information.

\begin{example}
\label{ex(7.1)}Let $\{X_{1},\ldots,X_{n}\}$ be sequence of $i.i.d.$
exponential random variables with probability density function
\[
f\left(  x;\lambda\right)  =\lambda e^{-\lambda x},\text{ \ \ \ }%
x>0,~\lambda>0.
\]

We simply have $MCKLE$ of $\lambda$ as%
\[
\widehat{\lambda}=\sqrt{\frac{2}{\bar{X^{2}}}}.
\]

This estimator is function of linear combinations of $x_{i}^{2}$'s, and so by
strong law of large numbers (SLLN), $\widehat{\lambda}$ is strongly consistent
for $\lambda$, as well as the $MME$ of $\lambda$.

Now, by CLT and delta method or using Theorem \ref{th(7.4)}, one can show that%
\[
\sqrt{n}\left(  \widehat{\lambda}-\lambda\right)  \overset{d}{\rightarrow
}N\left(  0,\frac{5\lambda^{2}}{4}\right)  ,
\]
and $n^{-1}$ order asymptotic bias of $\widehat{\lambda}$\ is $15\lambda/8n$.
It is well known that the $MLE$ of $\lambda$ is $\widehat{\lambda}_{m}%
=1/\bar{X}$ with asymptotic distribution%
\[
\sqrt{n}\left(  \widehat{\lambda}_{m}-\lambda\right)  \overset{d}{\rightarrow
}N\left(  0,\lambda^{2}\right)  ,
\]
and $n^{-1}$ order asymptotic bias of $\widehat{\lambda}_{m}$\ is $\lambda/n$.

Notice that using asymptotic bias of $\widehat{\lambda}$, we can find some
unbiasing factors to improve our estimator. Since the $MLE$ has inverse Gamma
distribution, the unbiased estimator of $\lambda$ is $\widehat{\lambda}%
_{um}=\left(  n-1\right)  /n\bar{X}$ \citep[see,][]{Forbes:et:al:2011}. In Liu
approach an approximately unbiased estimator of $\lambda$ is%
\begin{equation}
\widehat{\lambda}_{u}=\frac{8n}{8n+15}\sqrt{\frac{2}{\bar{X^{2}}}}.
\label{eq7.45}%
\end{equation}

Figure \ref{figExp} compares these estimators. In order to compare our
estimator and the $MLE$, we made a simulation study in which we used samples
of sizes $10$ to $55$ by $5$ with $10000$ repeats, where we assumed that the
true value of the model parameter is $\lambda_{true}=5$. The plots in Figure
\ref{figExp} show that the $MCKLE$ has more biased than the $MLE$, but $MCKLE$
in \eqref{eq7.45} which is approximately unbiased coincides with the unbiased
$MLE$.

\begin{figure}[ptb]
\begin{center}
{\includegraphics[height=3.5in,width=6in] {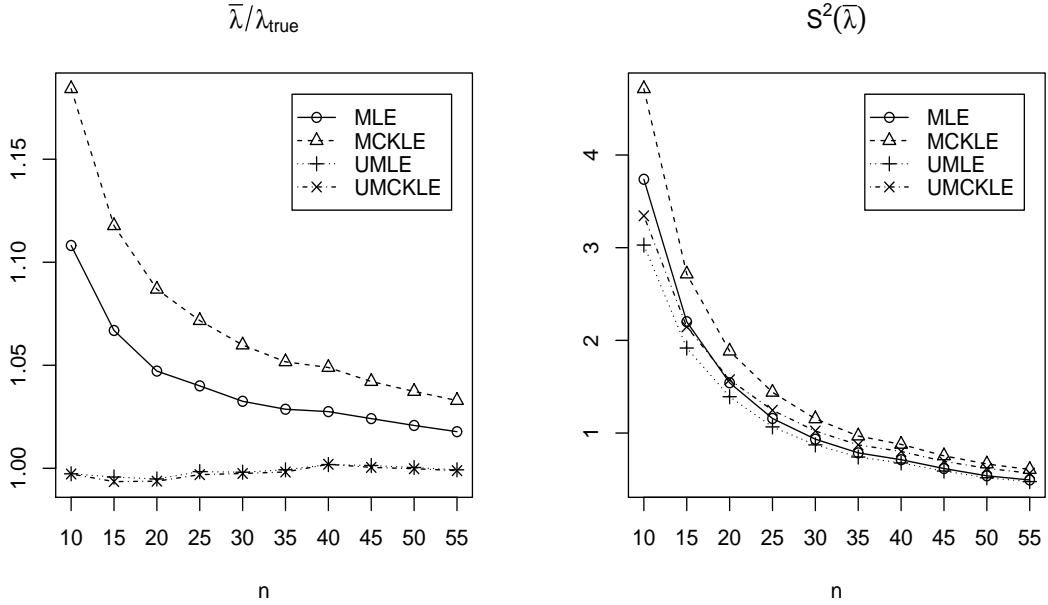}\\[0pt]}
\end{center}
\caption{$\bar{\lambda}/\lambda_{true}$ and $S^{2}\left(  \bar{\lambda
}\right)  $ as functions of sample size}%
\label{figExp}%
\end{figure}
\end{example}

\begin{remark}
In Example \ref{ex(5.2)}, note that $\left\vert X\right\vert $ has exponential
distribution. So, using Example \ref{ex(7.1)}, one can easily find asymptotic
properties of $\widehat{\theta}$ in Laplace distribution.
\end{remark}

\begin{example}
\label{ex(7.2)}Let $\{X_{1},\ldots,X_{n}\}$ be sequence of $i.i.d.$ two
parameter exponential random variables with probability density function
\[
f\left(  x;\mu,\sigma\right)  =\frac{1}{\sigma}e^{-\left(  x-\mu\right)
/\sigma},\text{ \ \ \ }x\geq\mu,~\mu\in%
\mathbb{R}
,~\sigma>0.
\]

It is not difficult to show that $MCKLE$ of $\mu$ and $\sigma$ are,
respectively,
\[
\widehat{\mu}=\bar{X}-\sqrt{\bar{X^{2}}-\bar{X}^{2}},~\widehat{\sigma}%
=\sqrt{\bar{X^{2}}-\bar{X}^{2}}.
\]

These estimators are functions of linear combinations of $x_{i}$'s and
$x_{i}^{2}$'s, and hence by SLLN, $\left(  \widehat{\mu},\widehat{\sigma
}\right)  $ are strongly consistent for $\left(  \mu,\sigma\right)  $, as well
as the $MME$ of $\left(  \mu,\sigma\right)  $.

Now, by CLT and delta method or using Theorem \ref{th(7.4)}, one can show that%
\[
V_{n}^{-1/2}\left(
\begin{array}
[c]{c}%
\widehat{\mu}-\mu\\
\widehat{\sigma}-\sigma
\end{array}
\right)  \overset{d}{\rightarrow}N_{2}\left(  0,I_{2}\right)  ,
\]
where%
\[
V_{n}=\frac{\sigma^{2}}{n}\left[
\begin{array}
[c]{cc}%
1 & -1\\
-1 & 2
\end{array}
\right]  .
\]

Figure \ref{figg2Exp} represents $g\left(  \mu,\sigma\right)  $ for a
simulated sample of size $100$ from two parameter exponential distribution
with parameters $\left(  \mu=3,\sigma=2\right)  $. The figure shows that the
estimators of $\mu$ and $\sigma$ are the values that minimize $g\left(
\mu,\sigma\right)  $.

\begin{figure}[ptb]
\begin{center}
{\includegraphics[height=3.96in] {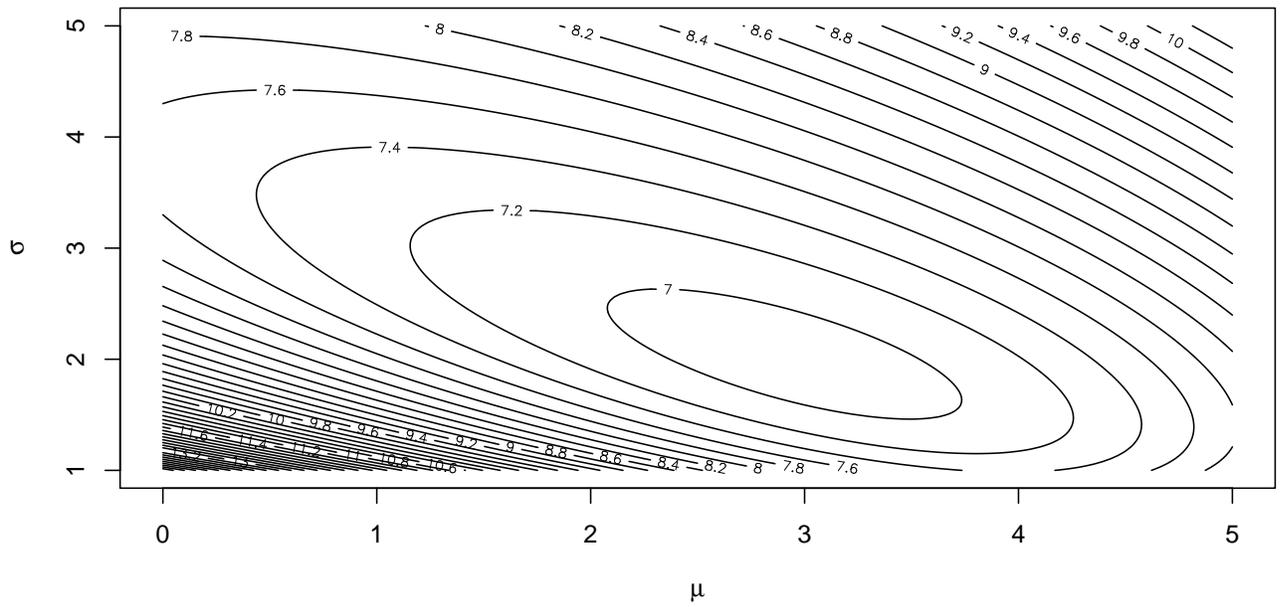}\\[0pt]}
\end{center}
\caption{$g\left(  \mu,\sigma\right)  $ for a simulated sample of size $100$
from two parameter exponential distribution with parameters $\left(
\mu=3,\sigma=2\right)  $}%
\label{figg2Exp}%
\end{figure}
\end{example}

\begin{example}
Let $\{X_{1},\ldots,X_{n}\}$ be sequence of $i.i.d.$ Pareto random variables
with probability density function
\[
f\left(  x;\alpha,\beta\right)  =\frac{\alpha\beta^{\alpha}}{x^{\alpha+1}%
},\text{ \ \ \ }x\geq\beta,~\alpha>0,~\beta>0.
\]

So we simply have%
\[
g\left(  \alpha,\beta\right)  =\frac{\alpha\beta}{\alpha-1}+\alpha\bar{x\log
x}-\alpha\bar{x}\left(  \log\beta+1\right)  +\alpha\beta,~\alpha>1.
\]

Differentiating $g\left(  \alpha,\beta\right)  $ with respect to $\beta$ and
setting zero gives
\[
\widehat{\beta}_{\alpha}=\frac{\bar{x}\left(  \alpha-1\right)  }{\alpha}.
\]

So, if we define the function $g$ of $\alpha$ as follows%
\[
g\left(  \alpha\right)  =g\left(  \alpha,\widehat{\beta}_{\alpha}\right)
=\alpha\bar{x}\log\frac{\alpha}{\alpha-1}+\alpha\bar{x\log x}-\alpha\bar
{x}\log\bar{x},~\alpha>1,
\]
then, derivative of $g\left(  \alpha\right)  $ with respect to $\alpha$ and
setting zero gives%
\[
\log\frac{\alpha}{\alpha-1}-\frac{1}{\alpha-1}+\frac{\bar{x\log x}}{\bar{x}%
}-\log\bar{x}=0.
\]

This equation can be solved numerically to find $MCKLE$ of parameters. Now,
using \ Theorem \ref{th(7.5)}, one can show that%
\[
V_{n}^{-1/2}\left(
\begin{array}
[c]{c}%
\widehat{\alpha}-\alpha\\
\widehat{\beta}-\beta
\end{array}
\right)  \overset{d}{\rightarrow}N_{2}\left(  0,I_{2}\right)  ,
\]
where%
\[
V_{n}=\frac{1}{n\left(  \alpha-2\right)  ^{3}}\left[
\begin{array}
[c]{cc}%
2\alpha\left(  \alpha-1\right)  ^{4} & \alpha\beta\left(  \alpha-1\right)
^{2}\\
\alpha\beta\left(  \alpha-1\right)  ^{2} & \frac{\beta^{2}}{\alpha}\left(
\alpha^{2}-2\alpha+2\right)
\end{array}
\right]  ,~\alpha>2.
\]

Figure \ref{figgPareto} represents $g\left(  \alpha,\beta\right)  $ for a
simulated sample of size $100$ from Pareto distribution with parameters
$\left(  \alpha=2,\beta=5\right)  $. This figure shows that the estimators of
$\alpha$ and $\beta$ are the values that minimize $g\left(  \alpha
,\beta\right)  $.

\begin{figure}[ptb]
\begin{center}
{\includegraphics[height=3.96in] {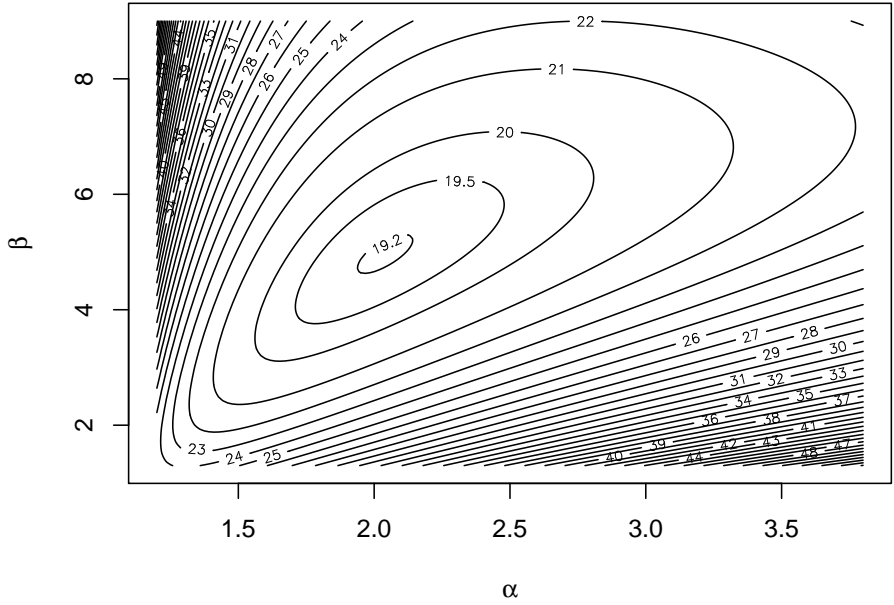}\\[0pt]}
\end{center}
\caption{$g\left(  \alpha,\beta\right)  $ for a simulated sample of size $100$
from Pareto distribution with parameters $\left(  \alpha=2,\beta=5\right)  $}%
\label{figgPareto}%
\end{figure}
\end{example}

\subsection{Asymptotic confidence interval}

In the following we assume that $\mbox{\boldmath $\theta$} $ is a scalar.
Using Theorem \ref{th(7.4)}, we can find an asymptotic confidence interval for
$\mbox{\boldmath $\theta$}$. Under the conditions of Theorem \ref{th(7.4)}, an
asymptotic $100\left(  1-\alpha\right)  \%$\ confidence interval for
$\mbox{\boldmath
$\theta$}$ is defined as%
\begin{equation}
P\left(  -z_{\frac{\alpha}{2}}<\frac{\sqrt{n}\left(  \widehat
{\mbox{\boldmath
$\theta$}}-\mbox{\boldmath
$\theta$}\right)  }{\sigma_{F}}<z_{\frac{\alpha}{2}}\right)  =1-\alpha,
\label{eq8.1}%
\end{equation}
where $z_{\alpha}$ is the $\left(  1-\alpha\right)  $-quantile of the
$N\left(  0,1\right)  $ and $\sigma_{F}$\ is defined in Theorem \ref{th(7.4)}.
If inequalities in \eqref{eq8.1} are not invertible, then we can use
$\widehat{\sigma}_{F}$ instead of $\sigma_{F}$ to obtain an approximate
confidence interval, where $\widehat{\sigma}_{F}$ is $\sigma_{F}$\ that
evaluated at $\mbox{\boldmath
$\theta$}=\widehat{\mbox{\boldmath
$\theta$}}$.

\cite{Pawitan:2001} presented an approach which is called likelihood interval
for parameters. Using his approach, one can find a divergence interval for the
parameter. Similar to the likelihood interval that is defined by
\cite{Pawitan:2001}, we define a divergence interval as a set of parameter
values with low enough divergence:%
\begin{equation}
\left\{  \mbox{\boldmath
$\theta$}\ s.t.\ \exp\left[  g\left(  \widehat{\mbox{\boldmath
$\theta$}}\right)  -g\left(  \mbox{\boldmath
$\theta$}\right)  \right]  >k\right\}  , \label{eq8.2}%
\end{equation}
for some cutoff point $k$, where $\exp\left[  g\left(  \widehat
{\mbox{\boldmath
$\theta$}}\right)  -g\left(  \mbox{\boldmath
$\theta$}\right)  \right]  $ is the normalized divergence with $g\left(
\mbox{\boldmath
$\theta$}\right)  $ as \eqref{eq2.4} or \eqref{eq5.8.4}; see
\citet[chapter 5]{Basu:et:al:2011}. Let us define the quantity $Q$ as%
\[
Q\left(  \widehat{\mbox{\boldmath
$\theta$}},\mbox{\boldmath
$\theta$}\right)  =\frac{2n\left[  g\left(  \mbox{\boldmath
$\theta$}\right)  -g\left(  \widehat{\mbox{\boldmath
$\theta$}}\right)  \right]  }{\sigma_{F}^{2}\cdot g^{\prime\prime}\left(
\widehat{\mbox{\boldmath
$\theta$}}\right)  },
\]
where%
\[
g^{\prime\prime}\left(  \mbox{\boldmath
$\theta$}\right)  =\frac{\partial^{2}}{\partial\mbox{\boldmath
$\theta$}^{2}}g\left(  \mbox{\boldmath
$\theta$}\right)  .
\]

Using Theorem \ref{th(7.4)}, we show that this quantity is asymptotically a
pivotal quantity. In other words, under the conditions of Theorem
\ref{th(7.4)},%
\begin{equation}
Q\left(  \widehat{\mbox{\boldmath
$\theta$}},\mbox{\boldmath
$\theta$}\right)  \overset{d}{\rightarrow}\chi_{1}^{2}. \label{eq8.5}%
\end{equation}

This is so, because using Taylor expansion of $g\left(  \mbox{\boldmath
$\theta$}\right)  $ around $\widehat{\mbox{\boldmath
$\theta$}}$, we have
\begin{equation}
Q\left(  \widehat{\mbox{\boldmath
$\theta$}},\mbox{\boldmath
$\theta$}\right)  \approx\frac{n\left(  \mbox{\boldmath
$\theta$}-\widehat{\mbox{\boldmath
$\theta$}}\right)  ^{2}}{\sigma_{F}^{2}}. \label{eq8.6}%
\end{equation}

Now using this fact, we can find the divergence interval for
$\mbox{\boldmath
$\theta$}$.

\begin{theorem}
\label{th(8.1)}Under the conditions of Theorem \ref{th(7.4)}, the asymptotic
$100\left(  1-\alpha\right)  \%$ divergence interval for $\mbox{\boldmath
$\theta$}$ is defined as \eqref{eq8.2}, with%
\[
k=\exp\left\{  -\frac{1}{2n}c\left(  \widehat{\mbox{\boldmath
$\theta$}}\right)  \chi_{\alpha,1}^{2}\right\}  ,
\]
where $\chi_{\alpha,1}^{2}$ is the $\left(  1-\alpha\right)  $-quantile of the
$\chi_{1}^{2}$ and%
\[
c\left(  \mbox{\boldmath
$\theta$}\right)  =\sigma_{F}^{2}\cdot g^{\prime\prime}\left(
\mbox{\boldmath
$\theta$}\right)  .
\]

\end{theorem}

\begin{proof} Using (\ref{eq8.5}), the probability that
divergence interval \eqref{eq8.2} covers $\mbox{\boldmath
$\theta$}$\ is%
\begin{eqnarray*}
P\left( \exp \left[ g\left( \widehat{\mbox{\boldmath
$\theta$}}\right) -g\left( \mbox{\boldmath
$\theta$}\right) \right] >k\right) &=&P\left( Q\left( \widehat{%
\mbox{\boldmath
$\theta$}},\mbox{\boldmath
$\theta$}\right) <-\frac{2n\log k}{\sigma _{F}^{2}\cdot g^{\prime \prime
}\left( \widehat{\mbox{\boldmath
$\theta$}}\right) }\right) \\
&=&P\left( \chi _{1}^{2}<-\frac{2n\log k}{\sigma _{F}^{2}\cdot g^{\prime
\prime }\left( \widehat{\mbox{\boldmath
$\theta$}}\right) }\right) .
\end{eqnarray*}
So, for some $0<\alpha <1$ we choose a cutoff%
\begin{equation*}
k=\exp \left\{ -\frac{1}{2n}\sigma _{F}^{2}\cdot g^{\prime \prime }\left(
\widehat{\mbox{\boldmath
$\theta$}}\right) \chi _{\alpha ,1}^{2}\right\} .
\end{equation*}
Since $\sigma _{F}^{2}$ is unknown, we estimate it with $\widehat{\sigma }%
_{F}^{2}$. This completes the proof. \end{proof}

\begin{remark}
Form \eqref{eq8.6}, The asymptotic confidence interval in \eqref{eq8.1} with
$\widehat{\sigma}_{F}$ instead of $\sigma_{F}$, is approximately equivalent
with that in \eqref{eq8.2}. Also, in \eqref{eq8.1} and \eqref{eq8.2}, we can
practically use sample version of $\sigma_{F}^{2}$ that is defined in Theorem
\ref{th(7.6)}.
\end{remark}

\begin{example}
In Example \ref{ex(7.1)}, the asymptotic $100\left(  1-\alpha\right)  \%$
divergence interval for $\lambda$ is in form \eqref{eq8.2} with%
\[
k=\exp\left\{  -\frac{5}{4n}\sqrt{\frac{\bar{X^{2}}}{2}}\chi_{\alpha,1}%
^{2}\right\}  .
\]

In other words, the confidence interval is in form%
\[
\left(  L,U\right)  =\frac{b\pm\sqrt{b^{2}-2\bar{X^{2}}}}{\bar{X^{2}}},
\]
with $b=-\log k+\sqrt{2\bar{X^{2}}}$. For a simulated sample of size $n=30$
from exponential distribution with parameter $\lambda=3$, Figure
\ref{figDCIexp} shows normalized divergence and asymptotic $95\%$\ confidence
interval for $\lambda$. In this typical sample $\bar{x^{2}}=0.2063127$,
$\widehat{\lambda}=3.113522$, $\widehat{\lambda}_{u}=2.930374$, $k=0.9498908$
and $\left(  L,U\right)  =\left(  2.092375,4.633022\right)  $.

\begin{figure}[ptb]
\begin{center}
{\includegraphics[height=3.5in] {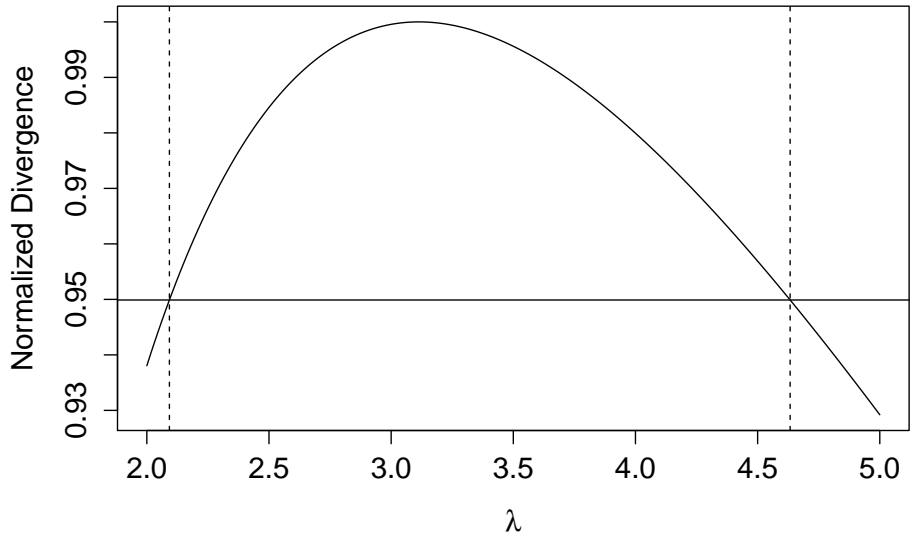}\\[0pt]}
\end{center}
\caption[Asymptotic $95\%$\ confidence interval for $\lambda\,$, in
$\exp\left(  \lambda=3\right)  $.]{Normalized divergence and asymptotic
$95\%$\ confidence interval for $\lambda\,$, in a simulated sample of size
$n=30$ from exponential distribution with parameter $\lambda=3$.}%
\label{figDCIexp}%
\end{figure}
\end{example}

\begin{remark}
When $\dim\left(  \mbox{\boldmath
$\theta$}\right)  >1$, we can't easily find a pivotal quantity. In these
cases, using quantiles of $g^{\ast}$ from repeated samples, we can find
cutoffs of divergence-based confidence regions.
\end{remark}

\begin{example}
In Example \ref{ex(7.2)}, Using $10000$ replicated simulated samples of size
$100$ from two parameter exponential distribution with parameters $\left(
\mu=3,\sigma=2\right)  $, we can find asymptotic cutoffs of divergence-based
confidence regions for $\left(  \mu,\sigma\right)  $. Figure \ref{figg2ExpCI}
shows asymptotic $90\%,70\%,\ldots,10\%$ confidence regions for $\left(
\mu,\sigma\right)  $. \begin{figure}[ptb]
\begin{center}
{\includegraphics[height=3.78in] {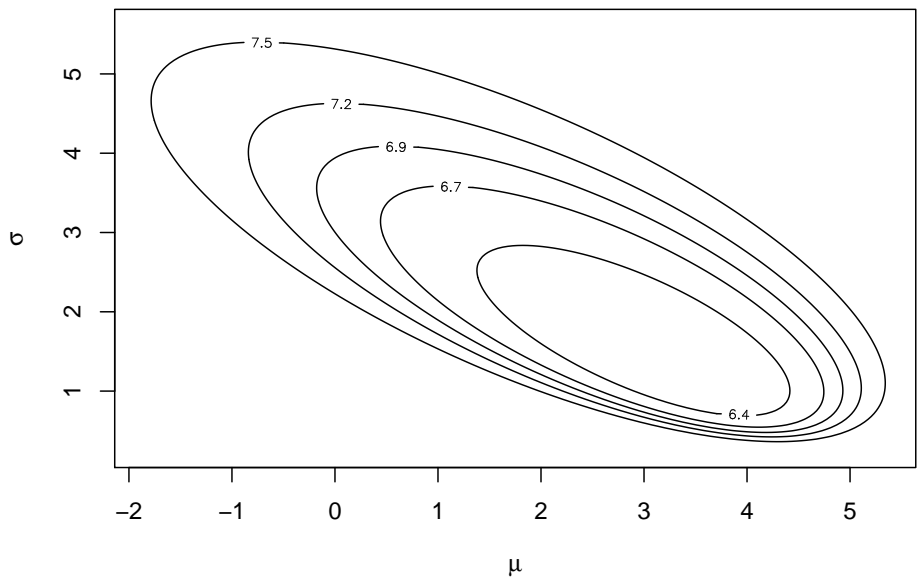}\\[0pt]}
\end{center}
\caption{Asymptotic $90\%,70\%,\ldots,10\%$ confidence regions for $\left(
\mu,\sigma\right)  $, using $10000$ replicated simulated samples of size $100$
from two parameter exponential distribution with parameters $\left(
\mu=3,\sigma=2\right)  $}%
\label{figg2ExpCI}%
\end{figure}
\end{example}

\subsection{Asymptotic hypothesis testing}

Let $\dim\left(  \mbox{\boldmath
$\theta$}\right)  =1$ and $\mathbf{\Theta}_{0}$ and $\mathbf{\Theta}_{1}$ be
two subsets of $\mathbf{\Theta}$ such that%
\[
\mathbf{\Theta}_{0}\cup\mathbf{\Theta}_{1}=\mathbf{\Theta,~\Theta}_{0}%
\cap\mathbf{\Theta}_{1}=\phi.
\]

We are interested in testing hypotheses%
\begin{equation}
H_{0}:\mbox{\boldmath
$\theta$}\in\mathbf{\Theta}_{0}\text{\ vs }H_{1}:\mbox{\boldmath
$\theta$}\in\mathbf{\Theta}_{1}. \label{eq9.2}%
\end{equation}

It is clear that by inverting asymptotic confidence interval in \eqref{eq8.1},
we can find a critical region for statistical tests (asymptotically of level
$\alpha$)%
\begin{equation}
H_{0}:\mbox{\boldmath
$\theta$}=\mbox{\boldmath
$\theta$}_{0}\text{\ vs }H_{1}:\mbox{\boldmath
$\theta$}\neq\mbox{\boldmath
$\theta$}_{0}, \label{eq9.3}%
\end{equation}
for a given $\mbox{\boldmath
$\theta$}_{0}$. Similar to approach of generalized likelihood ratio test,
\cite{Basu:1993} and \cite{Lindsay:1994} presented a divergence difference
test $\left(  DDT\right)  $ statistic based on \eqref{eq1.1} for testing
hypotheses in \eqref{eq9.3} in continuous and discrete cases; also see
\citet[chapter 5]{Basu:et:al:2011}. Here, we perform an alternative
statistical test based on $CKL$ divergence. For testing hypotheses in
\eqref{eq9.2}, we define the generalized divergence difference test $\left(
GDDT\right)  $ statistic as%
\begin{align*}
GDDT  &  =2n\left[  \underset{\mbox{\boldmath
$\theta$}\in\mathbf{\Theta}_{0}}{\inf}g\left(  \mbox{\boldmath
$\theta$}\right)  -\underset{\mbox{\boldmath
$\theta$}\in\mathbf{\Theta}}{\inf}g\left(  \mbox{\boldmath
$\theta$}\right)  \right] \\
&  =2n\left[  g\left(  \widehat{\mbox{\boldmath
$\theta$}}_{0}\right)  -g\left(  \widehat{\mbox{\boldmath
$\theta$}}\right)  \right]  .
\end{align*}

We consider behavior of $GDDT$ as a test statistic for a null hypothesis of
the form $H_{0}:\mbox{\boldmath
$\theta$}\in\mathbf{\Theta}_{0}$.

\begin{theorem}
\label{th(9.2)}Under the conditions of Theorem \ref{th(7.4)} and the null
hypothesis $H_{0}:\mbox{\boldmath
$\theta$}\in\mathbf{\Theta}_{0}$,%
\[
GDDT\overset{d}{\rightarrow}c\left(  \widehat{\mbox{\boldmath
$\theta$}}_{0}\right)  \chi_{1}^{2}.
\]

\end{theorem}

\begin{proof}
Using Taylor expansion of $g\left(  \widehat{\mbox{\boldmath
$\theta$}}_{0}\right)  $ around $\widehat{\mbox{\boldmath
$\theta$}}$ we get%
\begin{align*}
2n\left[  g\left(  \widehat{\mbox{\boldmath
$\theta$}}_{0}\right)  -g\left(  \widehat{\mbox{\boldmath
$\theta$}}\right)  \right]    & \approx2n\left[  g\left(  \widehat
{\mbox{\boldmath
$\theta$}}\right)  +\left(  \widehat{\mbox{\boldmath
$\theta$}}_{0}-\widehat{\mbox{\boldmath
$\theta$}}\right)  g^{\prime}\left(  \widehat{\mbox{\boldmath
$\theta$}}\right)  +\frac{\left(  \widehat{\mbox{\boldmath
$\theta$}}_{0}-\widehat{\mbox{\boldmath
$\theta$}}\right)  ^{2}}{2}g^{\prime\prime}\left(  \widehat{\mbox{\boldmath
$\theta$}}\right)  -g\left(  \widehat{\mbox{\boldmath
$\theta$}}\right)  \right]  \\
& =\frac{n\left(  \widehat{\mbox{\boldmath
$\theta$}}_{0}-\widehat{\mbox{\boldmath
$\theta$}}\right)  ^{2}}{2}g^{\prime\prime}\left(  \widehat{\mbox{\boldmath
$\theta$}}\right)  .
\end{align*}
Under $H_{0}$, the quantity $g^{\prime\prime}\left(  \widehat
{\mbox{\boldmath
$\theta$}}\right)  $ convergence to $g^{\prime\prime}\left(  \widehat
{\mbox{\boldmath
$\theta$}}_{0}\right)  $. Thus%
\[
GDDT\approx c\left(  \widehat{\mbox{\boldmath
$\theta$}}_{0}\right)  \frac{n\left(  \widehat{\mbox{\boldmath
$\theta$}}_{0}-\widehat{\mbox{\boldmath
$\theta$}}\right)  ^{2}}{\sigma_{F_{0}}^{2}},
\]
where $\sigma_{F_{0}}^{2}$ is $\sigma_{F}^{2}$ that evaluated at
$\mbox{\boldmath
$\theta$}=\widehat{\mbox{\boldmath
$\theta$}}_{0}$. Now using Theorem \ref{th(7.4)} the proof is complete.
\end{proof}

\begin{remark}
Under the conditions of Theorem \ref{th(9.2)}, we can obtain the following
approximation for the power function in a given $\mbox{\boldmath
$\theta$}_{1}\in\mathbf{\Theta}_{1}$ as%
\[
\beta\left(  \mbox{\boldmath
$\theta$}_{1}\right)  \approx P\left(  \chi_{1}^{2}>\frac{2n\left[  g\left(
\mbox{\boldmath
$\theta$}_{1}\right)  -g\left(  \widehat{\mbox{\boldmath
$\theta$}}_{0}\right)  \right]  +c\left(  \widehat{\mbox{\boldmath
$\theta$}}_{0}\right)  \chi_{\alpha,1}^{2}}{c\left(  \mbox{\boldmath
$\theta$}_{1}\right)  }\right)  .
\]

As an important application of the above approximation, one can find the
approximate sample size that guarantees a specific power $\beta$ for a given
$\mbox{\boldmath
$\theta$}_{1}\in\mathbf{\Theta}_{1}$. Let $n_{0}$ be the positive root of the
equation%
\[
\beta=P\left(  \chi_{1}^{2}>\frac{2n\left[  g\left(  \mbox{\boldmath
$\theta$}_{1}\right)  -g\left(  \widehat{\mbox{\boldmath
$\theta$}}_{0}\right)  \right]  +c\left(  \widehat{\mbox{\boldmath
$\theta$}}_{0}\right)  \chi_{\alpha,1}^{2}}{c\left(  \mbox{\boldmath
$\theta$}_{1}\right)  }\right)  ,
\]
i.e.,%
\[
n_{0}=\frac{c\left(  \mbox{\boldmath
$\theta$}_{1}\right)  \chi_{\beta,1}^{2}-c\left(  \widehat{\mbox{\boldmath
$\theta$}}_{0}\right)  \chi_{\alpha,1}^{2}}{2\left[  g\left(
\mbox{\boldmath
$\theta$}_{1}\right)  -g\left(  \widehat{\mbox{\boldmath
$\theta$}}_{0}\right)  \right]  }.
\]

The required sample size is then%
\begin{equation}
n^{\ast}=\left[  n_{0}\right]  +1, \label{eq9.15}%
\end{equation}
where $\left[  \cdot\right]  $ is used here to denote "integer part of".
\end{remark}

\begin{remark}
In special case that $\mathbf{\Theta}_{0}=\left\{  \mbox{\boldmath
$\theta$}_{0}\right\}  $, we can find a critical region for statistical tests
(asymptotically of level $\alpha$) in \eqref{eq9.3}. One can do this by
replacing $\widehat{\mbox{\boldmath
$\theta$}}_{0}$ with $\mbox{\boldmath
$\theta$}_{0}$.
\end{remark}

\begin{example}
In Example \ref{ex(7.1)}, the statistical test (asymptotically of level
$\alpha$) of null hypothesis $H_{0}:\lambda=\lambda_{0}$ against the
alternative $H_{1}:\lambda\neq\lambda_{0}$ is defined with the critical region%
\[
\bar{X^{2}}>\left(  \frac{b+\sqrt{b^{2}-4ac}}{2a}\right)  ^{2}\text{ or }%
\bar{X^{2}}<\left(  \frac{b-\sqrt{b^{2}-4ac}}{2a}\right)  ^{2},
\]
where $a=n\lambda_{0}^{2}$, $b=2\sqrt{2}n\lambda_{0}$ and $c=2n-\frac{5}%
{2}\chi_{\alpha,1}^{2}$.
\end{example}


\begin{thebibliography}{9999999999999999999999999999999}                                                                  %


\bibitem[Baratpour and Habibi Rad (2012)]{Baratpour:HabibiRad:2012}Baratpour,
S., \& Habibi Rad, A. (2012). Testing goodness-of-fit for exponential
distribution based on cumulative residual entropy. \textit{Communications in
Statistics-Theory and Methods}, \textbf{41(8)}, 1387-1396.

\bibitem[Basu (1993)]{Basu:1993}Basu, A. (1993). \textit{Minimum disparity
estimation: applications to robust tests of hypotheses.} Technical Report,
Center for Statistical Sciences, University of Texas at Austin.

\bibitem[Basu and Lindsay (1994)]{Basu:Lindsay:1994}Basu, A., \& Lindsay, B.
G. (1994). Minimum disparity estimation for continuous models: efficiency,
distributions and robustness. \textit{Annals of the Institute of Statistical
Mathematics}, \textbf{46(4)}, 683-705.

\bibitem[Basu et al. (2011)]{Basu:et:al:2011}Basu, A., Shioya, H., \& Park, C.
(2011). \textit{Statistical inference: the minimum distance approach}. CRC Press.

\bibitem[Broniatowski (2014)]{Broniatowski:2014}Broniatowski, M. (2014).
Minimum divergence estimators, maximum likelihood and exponential families.
\textit{Statistics \& Probability Letters}, \textbf{93}, 27-33.

\bibitem[Broniatowski and Keziou (2009)]{Broniatowski:Keziou:2009}%
Broniatowski, M., \& Keziou, A. (2009). Parametric estimation and tests
through divergences and the duality technique.\textbf{\ }\textit{Journal of
Multivariate Analysis,} \textbf{100(1)}, 16-36.

\bibitem[Cherfi (2011)]{Cherfi:2011}Cherfi, M. (2011). Dual $\phi$-divergences
estimation in normal models. \textit{arXiv preprint arXiv:}1108.2999.

\bibitem[Cherfi (2012)]{Cherfi:2012}Cherfi, M. (2012). Dual divergences
estimation for censored survival data. \textit{Journal of Statistical Planning
and Inference,} \textbf{142(7)}, 1746-1756.

\bibitem[Cherfi (2014)]{Cherfi:2014}Cherfi, M. (2014). On Bayesian estimation
via divergences. \textit{Comptes Rendus Mathematique}, \textbf{352(9)}, 749-754.

\bibitem[Fisher (1922)]{Fisher:1922}Fisher, R. A. (1922). On the mathematical
foundations of theoretical statistics. \textit{Philosophical Transactions of
the Royal Society of London. Series A, Containing Papers of a Mathematical or
Physical Character}, \textbf{222}, 309-368.

\bibitem[Forbes et al. (2011)]{Forbes:et:al:2011}Forbes, C., Evans, M.,
Hastings, N., \& Peacock, B. (2011). \textit{Statistical distributions}. John
Wiley \& Sons.

\bibitem[Hampel et al. (2011)]{Hampel:et:al:2011}Hampel, F. R., Ronchetti, E.
M., Rousseeuw, P. J., \& Stahel, W. A. (2011). \textit{Robust statistics: the
approach based on influence functions (Vol. 114)}. John Wiley \& Sons.

\bibitem[Huber (1964)]{Huber:1964}Huber, P. J. (1964). Robust estimation of a
location parameter. \textit{The Annals of Mathematical Statistics},
\textbf{35(1)}, 73-101.

\bibitem[Huber (1967)]{Huber:1967}Huber, P. J. (1967). The behavior of maximum
likelihood estimates under nonstandard conditions. \textit{Proceedings of the
fifth Berkeley symposium on mathematical statistics and probability.} 221-233.

\bibitem[Huber and Ronchetti (2009)]{Huber:Ronchetti:2009}Huber, P., \&
Ronchetti, E. (2009). \textit{Robust Statistics}, Wiley: New York.

\bibitem[Hwang and Park (2013)]{Hwang:Park:2013}Hwang, I., \& Park, S. (2013).
On scaled cumulative residual Kullback-Leibler information. \textit{Journal of
the Korean Data and Information Science Society}, \textbf{24(6)}, 1497-1501.

\bibitem[Jim\'{e}nz and Shao (2001)]{Jimenz:Shao:2001}Jim\'{e}nz, R., \& Shao,
Y. (2001). On robustness and efficiency of minimum divergence estimators.
\textit{Test}, \textbf{10(2)}, 241-248.

\bibitem[Lindsay (1994)]{Lindsay:1994}Lindsay, B. G. (1994). Efficiency versus
robustness: the case for minimum Hellinger distance and related methods.
\textit{The Annals of Statistics}, \textbf{22(2)}, 1081-1114.

\bibitem[Liu (2007)]{Liu:2007}Liu, J. (2007). \textit{Information theoretic
content and probability}. Ph.D. Thesis, University of Florida.

\bibitem[Morales et al. (1995)]{Morales:et:al:1995}Morales, D., Pardo, L., \&
Vajda, I. (1995). Asymptotic divergence of estimates of discrete
distributions. \textit{Journal of Statistical Planning and Inference},
\textbf{48(3)}, 347-369.

\bibitem[Park et al. (2012)]{Park:et:al:2012}Park, S., Rao, M., \& Shin, D. W.
(2012). On cumulative residual Kullback--Leibler information.
\textit{Statistics \& Probability Letters,} \textbf{82(11)}, 2025-2032.

\bibitem[Pawitan (2001)]{Pawitan:2001}Pawitan, Y. (2001). \textit{In all
likelihood: statistical modelling and inference using likelihood}: Oxford
University Press.

\bibitem[Qin and Lawless (1994)]{Qin:Lawless:1994}Qin, J., \& Lawless, J.
(1994). Empirical likelihood and general estimating equations. \textit{The
Annals of Statistics}, \textbf{22(1)}, 300-325.

\bibitem[Rohatgi and Saleh (2015)]{Rohatgi:Saleh:2015}Rohatgi, V. K., \&
Saleh, A. M. E. (2015). \textit{An introduction to probability and statistics
(2 ed.).} John Wiley. New York.

\bibitem[Serfling (1980)]{Serfling:1980}Serfling, R. (1980).
\textit{Approximation Theorems of Mathematical Statistics} John Wiley. New York.

\bibitem[Shao (2003)]{Shao:2003}Shao, J. (2003). \textit{Mathematical
Statistics (2 ed.)}. Springer, New York, USA.

\bibitem[van der Vaart (2000)]{vanderVaart:2000}van der Vaart, A. W. (2000).
\textit{Asymptotic statistics.} Cambridge university press.

\bibitem[Yari et al. (2013)]{Yari:et:al:2013}Yari, G., Mirhabibi, A., \&
Saghafi, A. (2013). Estimation of the Weibull parameters by Kullback-Leibler
divergence of Survival functions. \textit{Appl. Math}, \textbf{7(1)}, 187-192.

\bibitem[Yari and Saghafi (2012)]{Yari:Saghafi:2012}Yari, G., \& Saghafi, A.
(2012). Unbiased Weibull Modulus Estimation Using Differential Cumulative
Entropy. \textit{Communications in Statistics-Simulation and Computation},
\textbf{41(8)}, 1372-1378.
\end{thebibliography}
\end{document}